\documentclass[mathpazo]{cicp}

\usepackage[T1]{fontenc}
\usepackage[utf8]{inputenc}
\usepackage{ulem}
\usepackage[main=british, french, german, latin]{babel}
\usepackage{float}
\usepackage{lineno,hyperref}
\usepackage{xcolor}
\usepackage{verbatim}

\modulolinenumbers[5]
\usepackage{amssymb}
\usepackage{amsmath}
\newtheorem{pro}{Proposition}

\usepackage[percent]{overpic}
\usepackage{amsmath,amssymb}
\usepackage{lscape}

\begin{document}
%%%%% title : short title may not be used but TITLE is required.
% \title{TITLE}
% \title[short title]{TITLE}
\title{High order multiscale methods for advection-diffusion equation in highly oscillatory regimes: application to surfactant diffusion and generalization to arbitrary domains}

%%%%% author(s) :
% single author:
% \author[name in running head]{AUTHOR\corrauth}
% [name in running head] is NOT OPTIONAL, it is a MUST.
% Use \corrauth to indicate the corresponding author.
% Use \email to provide email address of author.
% \footnote and \thanks are not used in the heading section.
% Another acknowlegments/support of grants, state in Acknowledgments section
% \section*{Acknowledgments}
\author[C.~Astuto]{Clarissa Astuto\corrauth}
\address{Applied Mathematics and Computational Science, King Abdullah University of Science and Technology (KAUST),  4700, Thuwal, Saudi Arabia}
\email{{\tt clarissa.astuto@unict.it} (C.~Astuto)}

% multiple authors:
% Note the use of \affil and \affilnum to link names and addresses.
% The author for correspondence is marked by \corrauth.
% use \emails to provide email addresses of authors
% e.g. below example has 3 authors, first author is also the corresponding
%      author, author 1 and 3 having the same address.
% \author[Zhang Z R et.~al.]{Zhengru Zhang\affil{1}\comma\corrauth,
%       Author Chan\affil{2}, and Author Zhao\affil{1}}
% \address{\affilnum{1}\ School of Mathematical Sciences,
%          Beijing Normal University,
%          Beijing 100875, P.R. China. \\
%           \affilnum{2}\ Department of Mathematics,
%           Hong Kong Baptist University, Hong Kong SAR}
% \emails{{\tt zhang@email} (Z.~Zhang), {\tt chan@email} (A.~Chan),
%          {\tt zhao@email} (A.~Zhao)}
% \footnote and \thanks are not used in the heading section.
% Another acknowlegments/support of grants, state in Acknowledgments section
% \section*{Acknowledgments}

%%%%% Begin Abstract %%%%%%%%%%%
\begin{abstract}
In this paper, we propose high order numerical methods to solve a 2D advection diffusion equation, in the highly oscillatory regime. We use an integrator strategy that allows the construction of arbitrary high-order schemes {leading} to an accurate approximation of the solution without any time step-size restriction. 
This paper focuses on the multiscale challenges {in time} of the problem, that come from the velocity, an $\varepsilon-$periodic function, whose expression is explicitly known. $\varepsilon$--uniform third order in time numerical approximations are obtained. 
For the space discretization, this strategy is combined with high order finite difference
schemes. Numerical experiments show that the proposed
methods {achieve} the expected order of accuracy, and it is validated by several tests across diverse domains and  boundary conditions. The novelty of the paper consists of introducing a numerical scheme that is high order accurate in space and time, with a particular attention to the dependency on a small parameter in the time scale. The high order in space is obtained enlarging the interpolation stencil already established in \cite{COCO2013464}, and further refined in \cite{cocohigh}, with a special emphasis on the squared boundary, especially when a Dirichlet condition is assigned. In such case, we compute an \textit{ad hoc} Taylor expansion of the solution to ensure that there is no degradation of the accuracy order at the boundary. On the other hand, the high accuracy in time is obtained extending the work proposed in \cite{astuto2023time}. The combination  of high-order accuracy in both space and time is particularly significant due to the presence of two small parameters—$\delta$ and $\varepsilon$—in space and time, respectively. 
\end{abstract}
%%%%% end %%%%%%%%%%%

%%%%% AMS/PACs/Keywords %%%%%%%%%%%
%\pac{}
%\ams{52B10, 65D18, 68U05, 68U07}
\keywords{high order discretization, time multi-scale, advection-diffusion equation, surfactant, oscillating trap, arbitrary domains}

%%%% maketitle %%%%%
\maketitle

%%%% Start %%%%%%
\section{Introduction}
In this paper, we develop numerical schemes to solve a model of the diffusion of particles within a fluid in the presence of an oscillating flow, caused by the oscillation of a {specific} body or trap. This topic has practical applications, particularly in understanding the relationship between living cell membranes (acting {as} traps) and diffusing particles (for example vital substances). In such an application, an important factor would be the rate at which these substances are captured by the trap. To investigate the capture rate, a biomimetic model {was} created \cite{corti2014out,corti2015trapping,Raudino20168574}, where an oscillating air bubble emulates a fluctuating cell, and the flow of charged surfactants represents the diffusing substances (see Fig.~\ref{fig_delta} (a)). In this specific model, the surfactants consist of anions and cations with different configurations: the cations are hydrophilic, while the anions have a hydrophilic head and a hydrophobic tail, leading to their absorption at the air-water interface. 
%%%%%
The investigation of shape oscillations of drops immersed in a fluid with surfactants has been largely investigated, because of the multiple applications in nuclear physics, meteorology and chemical engineering \cite{rayleigh1882xx,cohen1962deformation,RENEKER20082387}. In \cite{lu1991shape}, the authors study the parameters that influence the oscillation frequency of an interface, with and without surfactants. 
\begin{figure}
	\centering
	\begin{minipage}[b]
		{.49\textwidth}
		\centering
		\begin{overpic}[abs,width=0.8\textwidth,unit=1mm,scale=.25]{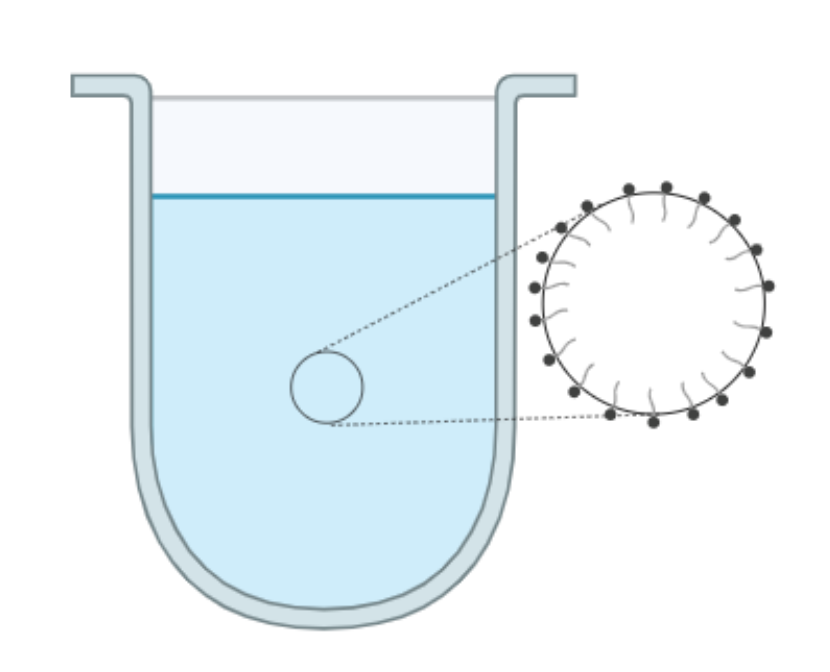}
			\put(-4,45){(a)}
		\end{overpic}
	\end{minipage}\hfill
	\begin{minipage}[b]{.49\textwidth}
		\begin{overpic}[abs,width=0.85\textwidth,unit=1mm,scale=.25]{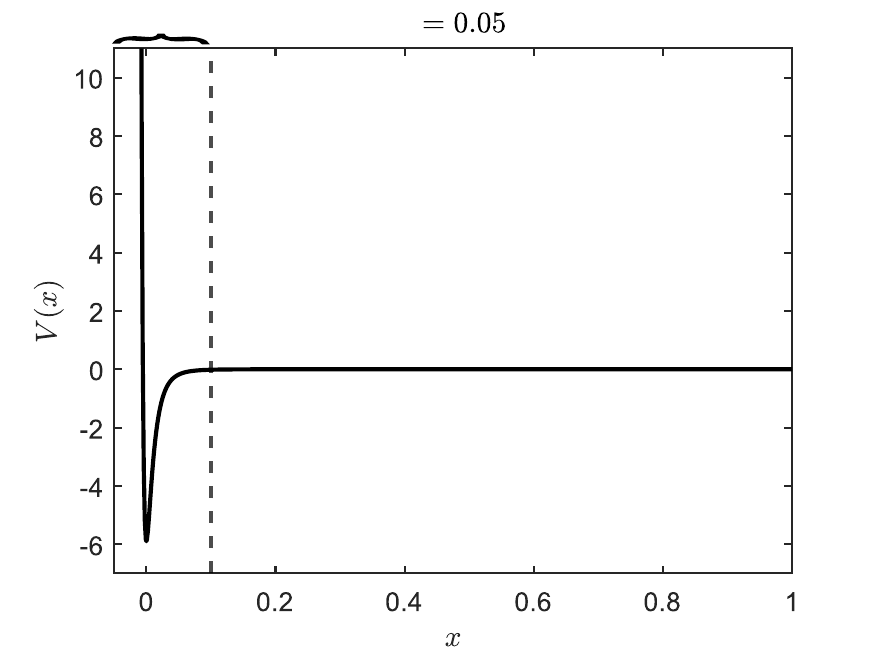}
			\put(27.5,44){{$\delta$}}
			\put(9,46){{$\Omega^\delta_{\rm b}$}}
			\put(45,45){{$\Omega_{\rm f}$}}
			\put(-1.,45){(b)}
		\end{overpic}
	\end{minipage}
	\caption{\textit{(a): Experimental setup of the diffusion-trapping of surfactants in presence of a trap. The zoom-in on the right shows the composition of the anions when they are stuck at the surface of the air bubble: the hydrophobic tails are inside the air bubble, while the hydrophilic heads lay on the surface. (b) Scheme of the potential $V(x)$, defined in Eq.~\eqref{eq_V_expr}, where $\delta$ is the thickness of the attractive-repulsive layer.}}
	\label{fig_delta}
\end{figure}
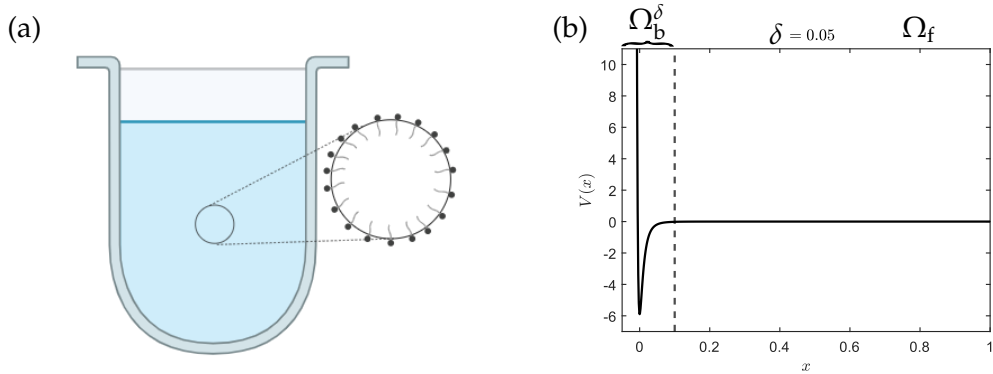

When studying the diffusion of particles in a moving fluid, it is fundamental to refer to advection-diffusion equations, that are important in many branches of engineering and applied sciences \cite{allen1982numerical,chatwin1985mathematical,kay1990advection,steefel2018approaches}. In this type of  {equation}, two main terms appear: a non-dissipative, advective, hyperbolic term and a dissipative, diffusive, parabolic one. Numerical methods generally perform well when diffusion dominates the equation. {However}, when advection prevails, undesirable phenomena, such as spurious oscillations or excessive numerical diffusion, may happen, and some stability property needs to be satisfied \cite{adler2023stable}. One possible approach to solve this issue is the implementation of fine mesh refinement, e.g., satisfying a suitable stability condition on the mesh P\'{e}clet number \cite{Wesseling2023600} (if we consider central differences instead of an upwind discretization), although it may not always be practical due to the significant increase {in} computational cost. Other approaches have been proposed, e.g. in \cite{de1955relaxation}, where a relaxation method, based on finite differences schemes, was developed for one-dimensional advection-diffusion equation, in the steady state regime, and in the presence of constant coefficients.

The main focus of this paper is the development of a high-order accurate numerical {method} in space and time to solve advection-diffusion equations with highly oscillatory boundaries, with application to surfactant diffusion in presence of a trap, and extension to domains of arbitrary shapes (see Fig.~\ref{fig_arbitrary_domains}). Low order numerical schemes are more commonly used because of their simpler implementation. However, when considering efficiency, even if high-order methods require greater programming efforts and more computational operations per grid point, they can significantly reduce the number of grid points needed to achieve a given error tolerance (see, e.g., Fig.~\ref{fig_cpu}). In multi-dimensional scenarios, this reduction can lead to substantial decrease of computational time and memory requirements, potentially by orders of magnitude. 

\begin{figure}
	\centering
	\begin{minipage}[b]
		{.33\textwidth}
		\centering
		\begin{overpic}[abs,width=1.\textwidth,unit=1mm,scale=.25]{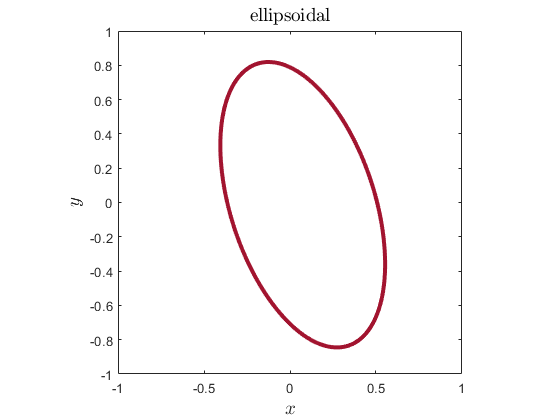}
			\put(4,34){(a)}
		\end{overpic}
	\end{minipage}\hfill
	\begin{minipage}[b]{.33\textwidth}
		\begin{overpic}[abs,width=\textwidth,unit=1mm,scale=.25]{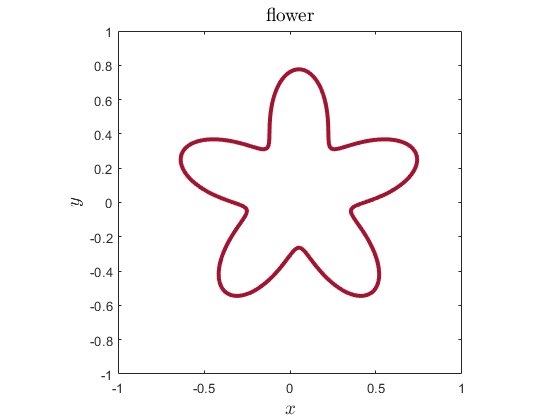}
			\put(4,34){(b)}
		\end{overpic}
	\end{minipage}
 \begin{minipage}[b]{.33\textwidth}
		\begin{overpic}[abs,width=\textwidth,unit=1mm,scale=.25]{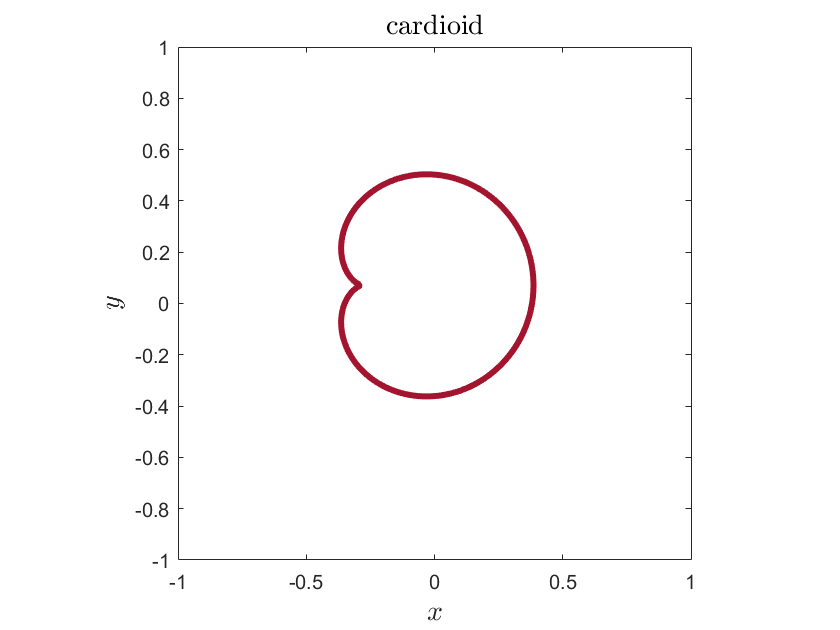}
			\put(4,34){(c)}
		\end{overpic}
	\end{minipage}
	\caption{\textit{Shape of different domains: (a) ellipsoidal, (b) flower-shaped and (c) cardioid-shape domain.}}
	\label{fig_arbitrary_domains}
\end{figure}

There are different classes of high order methods for solving time dependent convection dominated PDEs, from the high order weighted essentially non-oscillatory (WENO) schemes, capable of maintaining the robustness that is common to Godunov-type methods \cite{BALSARA2000405,shu2016high}, to the implicit-explicit (IMEX) Runge Kutta methods suitable for time  dependent partial differential systems which contain stiff and non stiff terms \cite{boscarino2016high,pareschi2000implicit,sebastiano2023high,astuto2023self}. Here, we are interested in solving an advection-diffusion equation, in the highly oscillatory regime. This work is part of a long time project, where the fluid velocity is a $\varepsilon$--periodic function (with $\varepsilon \ll 1$), that we suppose explicitly known. Numerous examples of oscillatory flows, both in terms of modeling and numerical treatment, can be encountered in the existing literature (\cite{allen1982numerical,chatwin_1975,10.1093/imamat/45.2.115,smith1982contaminant}).

When different time scales appear in the same equation, standard numerical methods produce errors of the order $\Delta t^p/\varepsilon^q$ (where $\Delta t$ is the time step), for some positive $p$ and $q$. To achieve the desired level of accuracy, there exists a restriction on the time step, i.e., $\Delta t \ll \varepsilon^{q/p}$, that becomes prohibitive for small values of $\varepsilon$. We will follow the strategy adopted in \cite{chartier2015uniformly,chartier2020new,chartier2022derivative,crouseilles2013asymptotic,crouseilles2017uniformly} although in the different contexts of Vlasov-Poisson equations, Klein-Gordon and nonlinear Schr\"odinger equations, to obtain a robust scheme that is able to deal with a large range of $\varepsilon \in (0,1]$ (being small or not), since our goal is to obtain a numerical scheme that is uniformly accurate in $\varepsilon$. We start from a first order scheme, as we did in \cite{astuto2023time}, and proceed to derive second and third order methods through recursive steps.

%We propose a third order in time scheme which leads to an accurate approximation of the solution without any time step-size restrictions.  Uniform in $\varepsilon$ time approximations are obtained with errors, and at a cost, that are independent of the oscillation frequency.

The problem of surfactants diffusion, that are adsorbed at the surface of a moving cell, has been investigated by several authors \cite{Raudino20168574,astuto2023multiscale,CiCP-31-707,SIAM2_ganesan2012arbitrary,SIAM3_morgan2015mathematical,SIAM4_xu2013analytical,WIEGEL1983283,BERG1977193} and the starting model for the evolution of single species carriers is the one described in \cite{astuto2023multiscale}, where the authors introduced the local concentration of 
ions $c= c(\vec{x},t)$, whose time evolution in a static fluid is governed by the conservation law
\begin{equation}
	\displaystyle \frac{\partial c}{\partial t}=-\nabla\cdot J,
	\label{equation_flux}
\end{equation}
where the expression for the flux $J$ is 
\[J = -D\left( \frac{\partial c}{\partial x} + c \frac{\partial V}{\partial x} \right),\]
$D$ is the diffusion coefficient, and the expression for the potential $V$ is described in Eq.~\eqref{eq_V_expr}. For the sake of simplicity, we will only discuss the one-dimensional model (see \cite{astuto2023time,astuto2023multiscale} for more details and dimensions). We assume that the fluid domain, which is not affected by the bubble, is $\Omega^\delta_{\rm f} = [L\delta,1]$, and that the attractive-repulsive mechanism of the bubble is simulated inside a thin region $\Omega^\delta_{\rm b}= [-\delta,L\delta]$, with $L$ a constant of order 1. 
It means that the potential $V(x)=0$ for $x \in \Omega^\delta_{\rm f}$.

Particles in the proximity of the bubble are initially attracted {to} its surface. However, when a particle gets extremely close to the surface, it experiences a repulsive {force} due to the impermeability of the air bubble. To simulate this behaviour, in \cite{astuto2023multiscale} we choose the Lennard-Jones potential (LJ) as a prototypical attractive-repulsive potential, as follows
\begin{eqnarray}
	\label{eq_V_expr}
	V(x) = E\left( \left(\frac{x+\delta}{\delta}\right)^{-12} - 2\left(\frac{x+\delta}{\delta}\right)^{-6} \right),
\end{eqnarray}
where $\delta$ denotes the range of the potential and $E$ represents the depth of the well, (see Fig.~\ref{fig_delta} (b)). Now we have all the ingredients to define the flux

In \cite{astuto2023multiscale}, we proposed a \textit{multiscale model}, employing an asymptotic expansion in $\delta$ to capture the dynamics of adsorption-desorption phenomena. The {spatial} multiscale nature derives from the potential that is not negligible only in $\Omega^\delta_{\rm b}$, which is very small compared to the entire domain. Summarizing, the 1D {multiscale model} for a single carrier, {under} the low concentration approximation, can be obtained for $\delta \ll 1$ (in general, for $\delta/L_x \ll 1$ where $L_x$ is the length of the 1D domain) as follows: 
\begin{align*}
	\frac{\partial  c }{\partial t} &= D\frac{\partial^2  c }{\partial x^2} \quad  x \in [0,1]\\
	\frac{\partial  c }{\partial x} &= 0  \quad  x = 1, \qquad \mathcal{M}\frac{\partial  c }{\partial t} = D\frac{\partial  c }{\partial x}  \quad  x = 0
\end{align*}
and
\begin{equation}
	\label{expr_M}
	\mathcal{M} = \delta\int_{0}^{L+1}\exp\left(-U(\zeta)\right)d\zeta.
\end{equation}
where $U(\zeta)= \phi \left( \zeta^{-12} - 2\zeta^{-6} \right)$ is a non dimensional form of the potential $V(x)$, with $\zeta = 1 + x/\delta \in [0,L+1]$ the rescaled variable, $L$ is the distance at which the potential $U$ is negligible and $\displaystyle \phi = {E}/{k_BT}$. In \cite{astuto2023multiscale}, we posed $L=2$. 

When considering the flow velocity, the role of oscillating traps becomes crucial {for calculating} their rates of adsorption and desorption at their surface. In our work \cite{ASTUTO2023111880}, {we} introduce an advection term into the diffusion equation {in the two-dimensional case} to account for fluid movement caused by the oscillations of the bubble.

In the presence of a moving fluid, the conservation law for the local concentration of ions $c = c(\vec{x},t)$ is the same as \eqref{equation_flux}
\begin{equation}
	\frac{
		\partial c}{\partial t}=-\nabla\cdot \vec J, \quad {\rm in }\,\,\Omega, \quad t \in [0, t_{\rm fin}]
\end{equation}
where $\mathcal{S} \subset \mathbb{R}^2$ is a rectangular domain, and $\Omega = \mathcal{S}\setminus \mathcal{B}$ the computational domain where $\mathcal{B}$ is a circle centered in $(0,0)$, with radius $R_{\mathcal{B}}$ (see Fig.~\ref{fig_discretization} (a)). However, this time the flux term $\vec J$ contains a diffusion and an advection term, 
\begin{equation} \label{eq:flux}
	\vec J=\ -D\nabla c - c\, \vec{u}, \quad {\rm in }\,\,\Omega
\end{equation}
where $t_{\rm fin}>0$, $\vec{u} = \vec{u}(\vec x,t/\varepsilon) \in \mathbb{R}^2$ is the explicitly known velocity, and it is assumed to be a periodic vector function of time with period equal to  $\varepsilon \in ]0, \varepsilon_0]$, for some $\varepsilon_0 > 0$. We add a subscript $\varepsilon$ on the concentration $c_\varepsilon = c$, to emphasize its dependence on the oscillation period, and at the end, the system reads
\begin{equation}
	\label{eq_ode}
	\frac{\partial c_\varepsilon}{\partial t} = D\Delta c_\varepsilon + \nabla \cdot (c_\varepsilon \vec{u}(t/\varepsilon)),  \quad {\rm in }\,\,\Omega.
\end{equation}
{From now on, we omit the explicit dependence of $\vec u$ in space, while we keep its dependence on $t/\varepsilon$.} The boundary of the domain is defined as $\Gamma = \partial \Omega = \Gamma_\mathcal{S} \cup \Gamma_\mathcal{B}$; see Fig.~\ref{fig_discretization} (a). 

Eq.~\eqref{eq_ode} is completed with homogeneous Neumann boundary conditions in $\Gamma_\mathcal{S}$ and absorption-desorption boundary conditions in $\Gamma_\mathcal{B}$ (see \cite{astuto2023multiscale} for more details), i.e., in other words
\begin{eqnarray}
	\displaystyle \nabla  c_\varepsilon \cdot n &=& 0 \quad  {\rm  on }\,\, \Gamma_\mathcal{S}\\ \label{eq_bc_M_2D}
	\displaystyle \mathcal M\frac{\partial  c_\varepsilon}{\partial t} &=& \mathcal M D \frac{\partial ^2 c_\varepsilon}{\partial \tau ^2}-D\frac{\partial  c_\varepsilon}{\partial n }\quad \text{ on $\Gamma_\mathcal{B}$},
\end{eqnarray}
%\giovanni{L'eq. (2.5) è sbagliata. Sarebbe corretta solo in 2D. Mi dispiace, ma i test numerici sono da rifare. L'operatore di Laplace-Beltrami sulla superficie della sfera in simmetria cilindrica diventa:
where $n$ is the outgoing normal vector to $\Gamma$, and $\tau$ is the tangent vector to $\Gamma_\mathcal{B}$. Eq.~\eqref{eq_bc_M_2D} is the analogue expression of Eq.~\eqref{expr_M}, but in higher dimension. 

To close the system~(\ref{eq_ode}--\ref{eq_bc_M_2D}), we add an initial condition 
\begin{eqnarray}
	\label{eq_IC}
	c _\varepsilon(t=0) &=&  c _\varepsilon^0 \quad {\rm in }\,\,\Omega
\end{eqnarray}
that does not depend on $\varepsilon$.

In this work, we focus our attention on the time multiple scale nature of the problem. %
Particularly relevant is the adsorption rate when the bubble is exposed to intense forced oscillations \cite{Raudino20168574}. The oscillation frequency is of the order of hundredths of \textit{Hz}, while the diffusing time is of the order of hours: these two different scales in time introduce in the model, as we already mentioned, multiscale challenges.

The plan of the paper is the following: in Section~\ref{section_numerical_scheme} we show different space discretizations, starting from a second order scheme for a drift-diffusion equation with constant coefficients, and going on with a fourth order numerical scheme for a drift-diffusion equation with variable coefficients. Next, we consider a more complicate system, introducing a circular hole in the domain, to be more realistic and closer to the main application of the paper. Regarding the space discretization, we apply a 9-point stencil (5 points for each space direction) fourth order numerical scheme \cite{Trottemberg:MG}, and in Section~\ref{section_space_discr_ghost} we show in details how to deal with the boundary conditions assigned in the circular hole. In Section~\ref{section_time_discr} we focus on the time derivative discretization, starting with a first and a second-order numerical schemes, that we already proposed in \cite{astuto2023time}, and we introduce a third order numerical scheme, showing how to derive it from the previous ones, with a recursive technique. In Section~\ref{section_results}, we show the accuracy tests of the numerical schemes that we have introduced, along with their applicability to new geometries. At the end we draw some conclusions.

\section{The numerical scheme{s}}
\label{section_numerical_scheme}
In this section, we first perform a second order space discretization, in a squared domain $\Omega = [a,b]^2$, and then we go on with higher order numerical schemes. The feature of the section is the following: we design a second order space discretization  for the advection-diffusion equation, starting with constant coefficients. The second step is the description of the 9-point stencil (5 points for each space direction) fourth order discretization, with variable coefficients. After {familiarizing the reader} with high order space discretizations in the case of a regular domain, we {proceed} to the description of the most interesting case of this {work:} a high order discretization in the presence of irregular domains and ghost points.

Next, we will present a first order implicit time discretization applied to Eq.~\eqref{eq_ode}, and discuss second and third order extensions afterwards.

\subsection{Space discretization}
\label{section_space_discr}
We use a uniform square Cartesian discretization, with $\Delta x = \Delta y = h$, and the set of grid points is ${\Omega}_h = (x_h,y_h) = \{(x_i,y_j)=(ih,jh), (i,j) \in \{0,\cdots,N\}^2 \}$, where $N \in \mathbb{N}$ and 
$h = L_x/N$ with $L_x = L_y = b-a$. We define as $\Gamma_h$ the set of boundary points, such that $\Gamma_h = \{ P = (x_i,y_j): \{i\in \{0,N\}, \, \forall j\} \cup \{j\in\{0,N\},\, \forall i\}\}$.
Let us consider Eq.~\eqref{eq_ode} when the advection term has a constant value for $u$. For simplicity, we drop the subscript $\varepsilon$, and it becomes
\begin{align}
	\label{eq_ode_constant}
	\frac{\partial c}{\partial t} = D\Delta c + u\nabla \cdot c  \quad {\rm in }\,\,\Omega,
\end{align}
where $D,u \in \mathbb R$. To consider a more general case, here we consider Dirichlet boundary conditions, such that
\begin{equation}
	\label{eq_Dir_bc}
	c = f, \quad {\rm on }\,\,\partial \Omega,
\end{equation}
where $f:\mathbb R^2 \to \mathbb R$.

In the following section we start with a second order discretization for the space derivatives.  

\subsubsection{Second order space discretization with constant coefficients}
%{\color{red}modificare senza column vector}
%If we represent the discretization of $ c$, $c^0 $ and $f$ as $ c_{h} = (\ldots,  c_{i,j}, \ldots), c^0_{h} = (\ldots,  c^{0}_{i,j}, \ldots),$ $ f_{h} = (\ldots,  \vec{f}_{i,j}, \ldots) \in \mathbb{R}^{N+1}$, t
The numerical solution at time $t$ is represented by the matrix $c_{h}:=c_h(t)$, whose components $c_{i,j}(t)$ are approximation of the exact solution on the grid points of $\Omega_h$,  i.e.\
$c_{i,j}(t)\approx c(t,x_i,y_j)$. 
The problem \eqref{eq_ode_constant} is then discretized in space %, leading to a linear system
\begin{align}
	\label{eq_2nd_space}
	\frac{\partial c_h}{\partial t} = L^{\rm 2nd}_h c_h + D_h^{\rm 2nd} c_h
\end{align}
where $L^{\rm 2nd}_h$ and $D_h^{\rm 2nd}$ %are $(N+1) \times (N+1)$ matrices 
represent the discretization of the space derivatives, defined as follows 
\begin{eqnarray} 
	\label{eq_Lh}
	L^{\rm 2nd}_h\, c_{h}\Big|_{i,j} &=& D \left( \frac{c_{i,j+1}  + c_{i,j-1} + c_{i+1,j}  + c_{i-1,j} - 4c_{i,j}}{h^2} \right)
	\\ \nonumber D_h^{\rm 2nd}\, c_{h}\Big|_{i,j} & = & u \left( \frac{c_{i+1,j} - c_{i-1,j} + c_{i,j+1} - c_{i,j-1}}{2h}  \right).	
\end{eqnarray}
Regarding the boundary conditions, we assume that $\forall P\in \Gamma_h$, we have that
\begin{equation}
	L^{\rm 2nd}_h\, c_{h}\Big|_{P} = 1\cdot c_{h} \Big|_{P} = f_h,
\end{equation}
where $f_h$ is the approximation of boundary condition on the grid points of $\Omega_h$. At the end, we consider the solution $c_{h}$ as a long column vector, and the operators $L^{\rm 2nd}_h$ and $Q^{\rm 2nd}_h$ are represented by $(N+1) \times (N+1)$,
large, sparse block  matrices.

After this description, we go to higher order accurate discretizations, and in the next section we describe a fourth order numerical scheme.
\subsubsection{Fourth order space discretization and boundary conditions}
\label{section_discr_space}
For the fourth order spatial discretization, we consider a formula that uses a 9-point stencil, (5 points in each space direction) \cite{leveque1998finite,boscarino2019high,boscarino2022high}. In this case, Eq.~\eqref{eq_2nd_space} becomes
\begin{align}
	\label{eq_4th_space}
	\frac{\partial c_h}{\partial t} = L^{\rm 4th}_h c_h + D_h^{\rm 4th} c_h
\end{align}
where
% \begin{align}
%     \partial^2_{xx}\, c\Big|_{x_j} \approx \frac{1}{12 h^2}(-c_{j-2} + 16\,c_{j-1} -30\,c_j + 16\,c_{j+1} -c_{j+2}), 
% \end{align}
% thus, in 2D it becomes 
\begin{align}
	\label{eq_L4th}
	L^{\rm 4th}_h\, c_{h}\Big|_{i,j} = & \frac{1}{12 h^2}(-c_{i,j-2} + 16\,c_{i,j-1}  + \\ & \nonumber 16\,c_{i,j+1} -c_{i,j+2} -c_{i-2,j} + 16\,c_{i-1,j}  + 16\,c_{i+1,j} -c_{i+2,j} -60\,c_{i,j}),
\end{align}
and 
\begin{equation}
	D_h^{\rm 4th}\, c_{h}\Big|_{i,j} = \frac{u}{12h}(c_{i,j-2} -8\,c_{i,j-1}  + 8\,c_{i,j+1} -c_{i,j+2} + c_{i-2,j} -8\,c_{i-1,j}  + 8\,c_{i+1,j} -c_{i+2,j})
\end{equation}
where $L_h^{\rm 4th}$ and $D_h^{\rm 4th}$ are a $(N+1) \times (N+1)$ matrices.

In the case of  $\{i,j\} \in \{i\in \{1,N-1\}, \, \forall j\} \cup \{j\in\{1,N-1\},\, \forall i\}$, that are those points in which at least one of the 9-point stencil lays outside of the grid points set, $\Omega_h$, we consider the following approximation. For the sake of simplicity, let us start from the following Taylor expansions of the function $c$ in 1D space dimension:
\begin{align*}
	c(x + h) &= c(x) + hc'(x) + \frac{h^2}{2} c''(x) + \frac{h^3}{3!}c'''(x)+ O(h^4)\\
	c(x - h) &= c(x) - hc'(x) + \frac{h^2}{2} c''(x) - \frac{h^3}{3!}c'''(x)+ O(h^4) \\
	c(x + 2h) &= c(x) + (2h)c'(x) + \frac{(2h)^2}{2} c''(x) + \frac{(2h)^3}{3!}c'''(x)+ O(h^4)\\
	c(x - 2h) &= c(x) - (2h)c'(x) + \frac{(2h)^2}{2} c''(x) - \frac{(2h)^3}{3!}c'''(x)+ O(h^4).
\end{align*}
% \begin{figure}[!ht]
% 	\centering
% 	% \begin{overpic}[abs,width=0.5\textwidth,unit=1mm,scale=.25]{Figures/V_x_delta_05.pdf}
% 	% \put(23.5,38){{$\delta$}}
% 	% \put(0,39){{attract.-repuls.}}
% 	% \end{overpic}
% 	\includegraphics[width=0.5\textwidth]{Figures/accuracy_Dh_Lh_4th.png}
% 	\caption{\textit{Accuracy error in $L^1,L^2,L^\infty$-norms, with $t_{\rm fin } = 0.1,$ for a fixed $\Delta t_{\rm ref} = 10^{-5}$ for Eq.~\eqref{eq_4th_space}. The domain is $\Omega = [-1,1]^2$,  $N_{\rm ref} = 640$, the initial condition is defined in Eq.~\eqref{eq_expr_IC_tests}  with $x_{m_1} = y_{m_1} = 0, \sigma = 0.1$ and homogeneous Dirichlet boundary conditions (i.e., $f = 0$ in Eq.~\eqref{eq_Dir_bc}).}}
% 	\label{fig_4th_space}
% \end{figure}
After some algebraic computation we have 
\[
c(x) = \frac{4\, c(x+2h) + 4\,c(x-2h) - c(x+h) - c(x-h)}{6},
\]
that at discrete levels becomes
\[
c_j = \frac{4\, c_{j+2} + 4\,c_{j-2} - c_{j+1} - c_{j-1}}{6},
\]
and from this formula we can extend it to 2D, and extrapolate those values ${c}_{i,j}, \{i,j\} \in \{i\in \{1,N-1\}, \, \forall j\} \cup \{j\in\{1,N-1\},\, \forall i\}$. For example, regarding the operator $L^{\rm 4th}_h$, if $i = 2$ and $j = N-1$, we have
\begin{align*}
	L^{\rm 4th}_h\, c_{h}\Big|_{2,N-1} = \frac{1}{12h^2}\left(12\,c_{1,N-1} + 12\,c_{3,N-1} + 12\,c_{2,N-2} + 12\,c_{2,N} - 48\,c_{2,N-1} \right).
\end{align*}

\subsubsection{Fourth order space discretization with variable coefficients}
\label{section_variable_coeff}
In this section, we add the dependence in space and time of the velocity $\vec u$, as our paper focuses on describing the time-oscillating movement of an obstacle (the air bubble in Fig.~\ref{fig_delta} (a)). In this case, Eq.~\eqref{eq_ode_constant} becomes
\begin{align}
	\label{eq_ode_variable}
	\frac{\partial c}{\partial t} = D\Delta c + \nabla \cdot(\vec u c)  \quad {\rm in }\,\, \Omega,
\end{align}
where $\vec u = \vec u(\vec x,t)$ is a known function.

We approximate ${\vec u}$ on the grid points $\Omega_h$, as $\vec{u}_{h}(t) = \{ \vec{u}_{i,j}(t): i\in \{0,N\}, j\in\{0,N\}\}$, where $\vec{u}_{i,j}(t) = [{u}^x_{i,j}(t),{u}^y_{i,j}(t)]$. {In problem \eqref{eq_ode_variable}, the} advection term is then discretized in space, leading to a new linear system 
\begin{align}
	\label{eq_4th_space_variable}
	\frac{\partial c_h}{\partial t} = L^{\rm 4th}_h c_h + {Q^{\rm 4th}_h}(\vec{u}_h) c_h
\end{align}
where $Q_h^{\rm 4th}(\vec{u}_h)$ is a $(N+1) \times (N+1)$ matrix, representing the discretization of the advection operator, and its product with the solution is defined as follows
\begin{align} \label{eq_Q4th}
	Q_h^{\rm 4th}\,\left( \vec{u}_h\right)\,c_{h}\Big|_{i,j} = {\frac{1}{12h}}&
	\begin{pmatrix}
		u^x_{i-2,j} & -8u^x_{i-1,j} & 8u^x_{i+1,j} & -u^x_{i+2,j}
	\end{pmatrix} \cdot \begin{pmatrix}
		c_{i-2,j} \\ c_{i-1,j} \\ c_{i+1,j} \\ c_{i+2,j} 
	\end{pmatrix} + \\ {\frac{1}{12h}} &
	\begin{pmatrix}
		u^y_{i,j-2} & -8u^y_{i,j-1} & 8u^y_{i,j+1} & -u^y_{i,j+2}
	\end{pmatrix} \cdot \begin{pmatrix}
		c_{i,j-2} \\ c_{i,j-1} \\ c_{i,j+1} \\ c_{i,j+2} 
	\end{pmatrix}, \nonumber
\end{align}
where we drop the explicit dependence in time of the components of $\vec u_h$, for simplicity.

\subsection{Irregular domain and ghost points}
\label{section_space_discr_ghost}
In this section, we describe the space discretization for Eqs.~(\ref{eq_ode}--\ref{eq_IC}). The domain is \\ $\Omega = ([-L_x/2,L_x/2]\times [-L_y/2,L_y/2])\setminus \mathcal{B}$, with $\mathcal{B}$ a circle centered in $(0,0)$ and radius $R_{\mathcal{B}}$ (see Fig.~\ref{fig_discretization} (a)), and the problem reads:
\begin{eqnarray} 
	\left\{
	\begin{array}{l}
		\displaystyle     \frac{\partial c_\varepsilon}{\partial t} = D\Delta c_\varepsilon + \nabla \cdot (c_\varepsilon \vec{u}(t/\varepsilon))  \quad \text{ in $\Omega$} \\
		\displaystyle \nabla  c_\varepsilon \cdot n = 0 \quad  \text{ on $\Gamma_\mathcal{S}$}\\ \label{eq:bc_ghost}
		\displaystyle \mathcal{M}\frac{\partial  c_\varepsilon}{\partial t} = 
		\mathcal{M} D \Delta_\perp c_\varepsilon
		%        \frac{\partial ^2 c_\varepsilon}{\partial \tau ^2}
		-D\frac{\partial  c_\varepsilon}{\partial n }\quad \text{ on $\Gamma_\mathcal{B}$}
	\end{array}
	\right.
\end{eqnarray}
where the expression for the velocity $\vec{u}(t/\varepsilon) = \vec{u}(x,y,t/\varepsilon)$ is known, $n$ is the outgoing normal vector to $\Gamma_\mathcal{B}$, and $\Delta_\perp = {\partial^2}/{\partial \tau^2}$ denotes the Laplace-Beltrami operator on the circumference of the circle
(see Fig.~\ref{fig_discretization} (a)).  
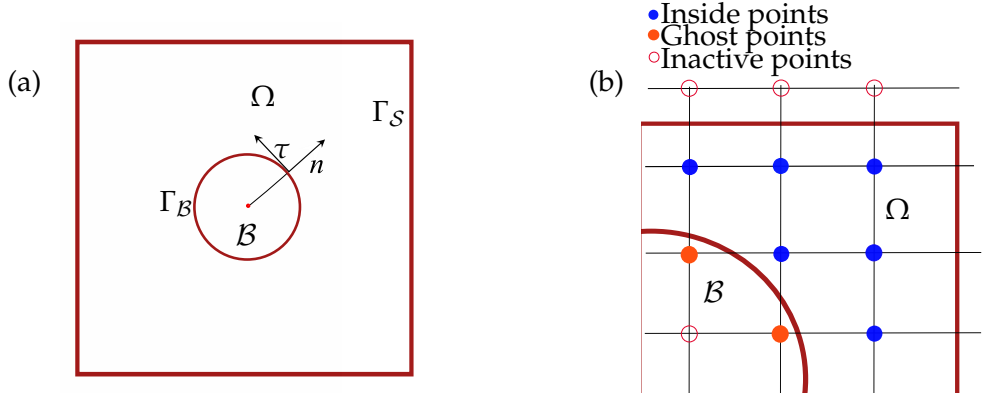
\begin{figure}[h]
	\centering
	\begin{minipage}[b]
		{.49\textwidth}
		\centering
		\begin{overpic}[abs,width=0.65\textwidth,unit=1mm,scale=.25]{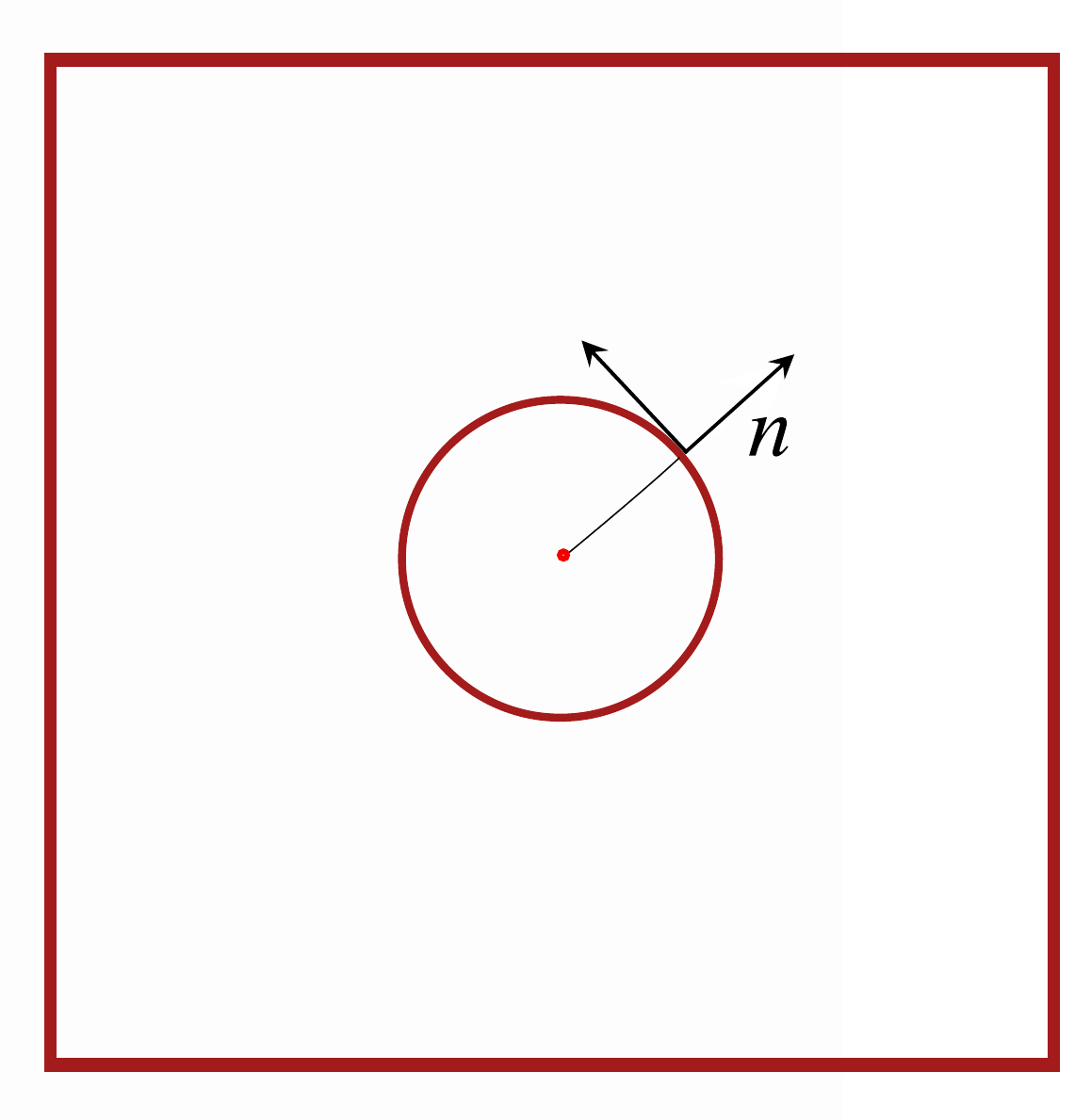}
			\put(-7.,40){(a)}
                \put(28,31){$\tau$}
			\put(23,19.5){$\mathcal{B}$}
			\put(25,38){$\Omega$}
			\put(41,36){$\Gamma_\mathcal{S}$}
			\put(13,24){$\Gamma_\mathcal{B}$}
		\end{overpic}	%\includegraphics[width=0.7\textwidth]{domains3D_squared}
	\end{minipage}\hfill
	\begin{minipage}[b]
		{.49\textwidth}
		\centering
		\begin{overpic}[abs,width=0.65\textwidth,unit=1mm,scale=.25]{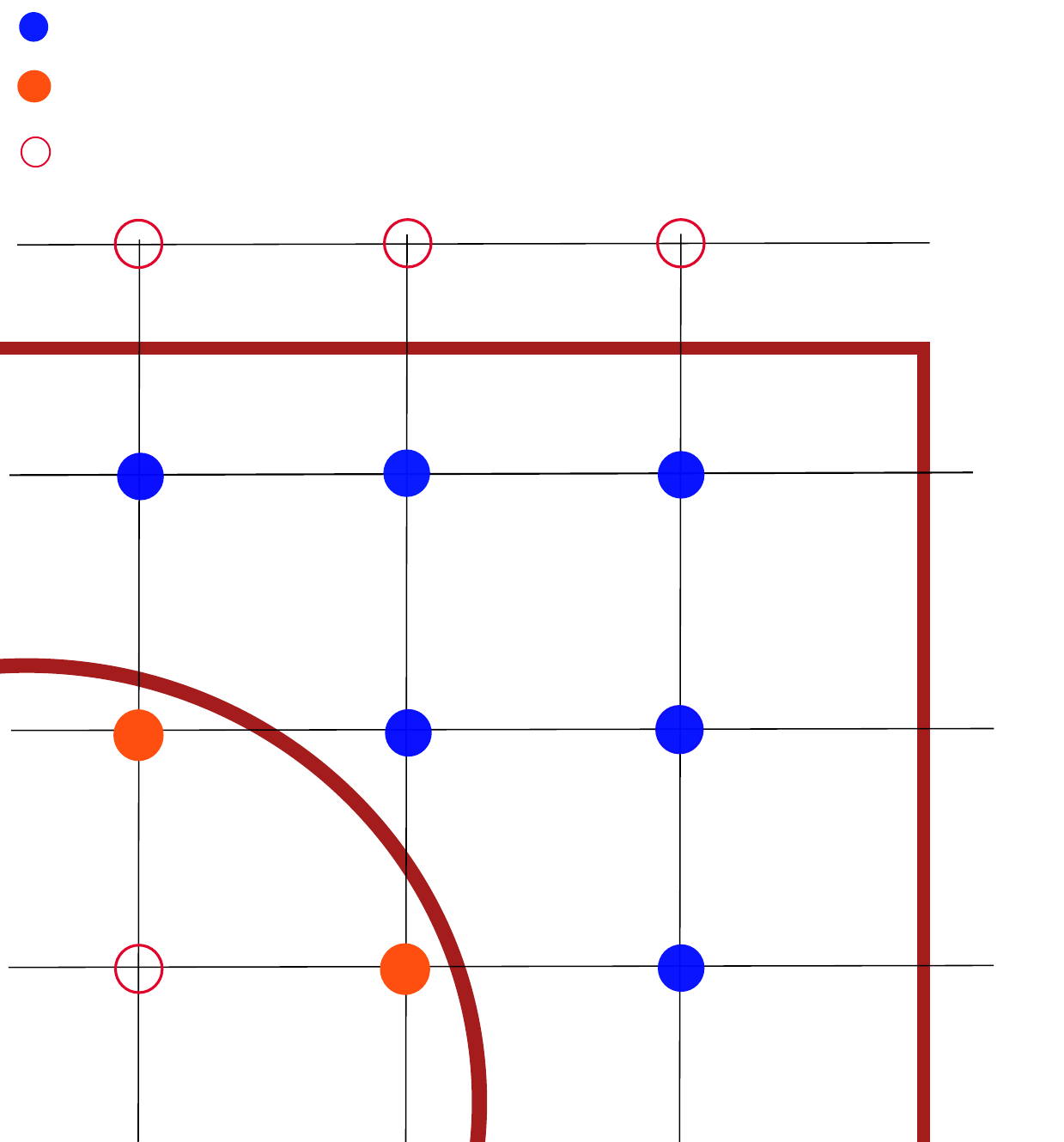}
			\put(2.5,43){{Inactive points}}
			\put(2.5,46){{Ghost points}}
			\put(2.5,49){{Inside points}}
			\put(-7.,40){(b)}
			\put(8,12){$\mathcal{B}$}
			\put(32,23){$\Omega$}
		\end{overpic}	%\includegraphics[width=0.75\textwidth]{classification_points3D}
	\end{minipage}
	\caption{\textit{(a): Representation of the domain $\Omega$ in 2D, where $\Gamma_\mathcal{S}$ is the external wall, $\mathcal{B} $ is the bubble with boundary $\Gamma_{\mathcal{B}}$ and radius $R_\mathcal{B}$. The vectors $n$ and $\tau$ are the outer and tangential unit vectors to $\Gamma_B$, respectively. (b): Classification of the inside (blue circles), ghost (orange circles) and inactive (red hole circles) points.}}
	\label{fig_discretization}
\end{figure}
\subsubsection{High order and ghost points}
The space discretization of the internal points of the domain $\Omega$ (see Fig.~\ref{fig_discretization} (a)) is analogue to the one {in} Section~\ref{section_space_discr}, with a particular attention to the boundary points close to $\Gamma_\mathcal{B}$. 
As we mentioned in the Introduction, the focus of this paper is {on achieving} uniform high order time discretization, and, in Section~\ref{section_time_discr}, we introduce a third order numerical scheme {for time integration}. Therefore, in this section, we propose a third order accurate spatial discretization for the ghost points. To obtain such an order, we need a 16-point stencil for the interpolation of the numerical solution at the ghost point $G$ (see  Fig.~\ref{stencil} (b)). Here we remark {that,} to obtain third order accuracy in {space,} we need a 16-point stencil due to the presence of first and second derivatives in the boundary {conditions.} However, in simpler scenarios like those with Dirichlet boundary conditions, such as in system \eqref{eq:Dirichlet_arbitrary}, the same stencil can achieve fourth order accuracy.

Following the approach showed in  \cite{sussman1994level,Osher,russo2000remark,book:72748}, the bubble $\mathcal{B}$ is implicitly defined by a level set function $\phi(x,y)$ that is positive inside the bubble, negative outside and zero on the boundary $\Gamma_\mathcal{B}$ :
\begin{eqnarray}
	\mathcal{B} = \{(x,y): \phi(x,y) > 0\}, \qquad
	\Gamma_\mathcal{B} = \{(x,y): \phi(x,y) = 0\}.
\end{eqnarray}
The unit normal vector $n$ in \eqref{normaleq} can be computed as $n = \frac{\nabla \phi }{|\nabla \phi|}$.
%\begin{equation}\label{LSnormal}
%	n = \frac{\nabla \phi }{|\nabla \phi|}
%\end{equation}
For a spherical bubble $\mathcal{B}$ centered at the origin, the most convenient level-set function, in terms of numerical stability, is {the} signed distance function between $(x,y)$ and $\Gamma_\mathcal{B}$, i.e.\ $\phi=R_\mathcal{B}-\sqrt{x^2+y^2}$.

After defining a uniform square Cartesian discretization, such that \\ ${\Omega}_h = (x_h,y_h) = \{(x_i,y_j)=(ih,jh), (i,j) \in \{0,\cdots,N\}^2 \}$,  with $h = \Delta x = \Delta y = Lx/(N-1), \, N\in \mathbb N$, and the set of grid points $\mathcal{S}_h$, we define the set of internal points $\Omega_h = \mathcal{S}_h \cap \Omega$, the set of bubble points $\mathcal{B}_h =\mathcal{S}_h \cap \mathcal{B}$ and the set of ghost points $\mathcal{G}_h$, which are points that belong to $\mathcal{B}$, with at least an internal point as neighbor, and are formally defined as follows
\begin{equation}
	(x_i,y_j) \in \mathcal{G}_h \iff (x_i,y_j) \in \mathcal{B}_h \text{ and } \{(x_i \pm h,y_j),(x_i,y_j\pm h) \} \cap \Omega_h \neq \emptyset.
\end{equation}
The other grid points, $\mathcal{S}_h\setminus(\Omega_h \cup \mathcal{G}_h)$, are called inactive points. See Fig.~\ref{fig_discretization} (b) for a classification of the points that belong to $S_h$.
Let $N_I = |\Omega_h|$ and $N_G = |\mathcal{G}_h|$ be the cardinality of the sets $\Omega_h$ and $\mathcal{G}_h$, respectively, and $\mathcal{N} = N_I + N_G$ the total number of active points. 
To compute the solution $ c_{\varepsilon,h} $ at the grid points of $\Omega_h \cup \mathcal{G}_h$, we employ a finite difference discretization of the equations at the $N_I$ internal grid points, along with appropriate interpolations for the $N_G$ ghost values to complete the system. These points play a crucial role in discretizing the equations for the internal points that are in the proximity to $\Gamma_\mathcal{B}$. At the end, the equations for the ghost points are coupled with the ones for the internal points, and the result is a $\mathcal{N} \times \mathcal{N}$ system, with non-eliminated boundary conditions. 
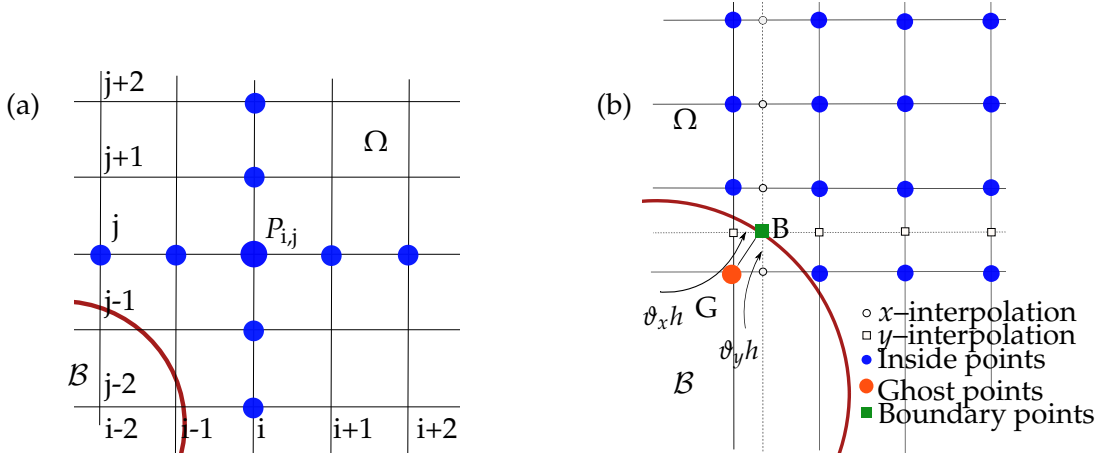
\begin{figure}[H]
	\centering
	\hfill
	\begin{minipage}[b]
		{.45\textwidth}
		\centering
		\begin{overpic}[abs,width=0.75\textwidth,unit=1mm,scale=.25]{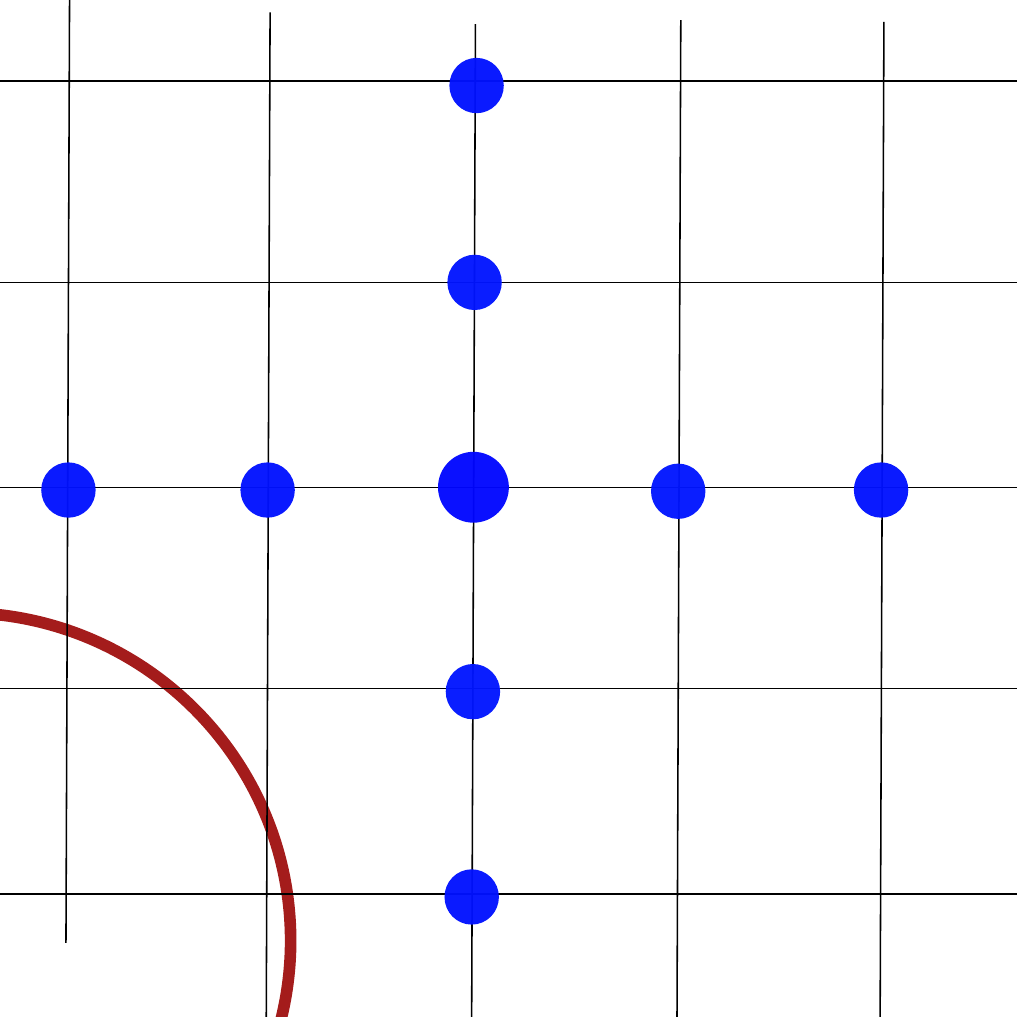}
			\put(38,40){$\Omega$}
			\put(-1,9){$\mathcal{B}$}
			\put(-9.,45){(a)}
			\put(4.,48){j+2}
			\put(4.,38){j+1}
			\put(5,29){j}
			\put(25,28.5){$P_{\rm i,j}$}
			\put(4,19){j-1}
			\put(4,8){j-2}
			\put(4,2){i-2}
			\put(14,2){i-1}
			\put(24,2){i}
			\put(34,2){i+1}
			\put(45,2){i+2}
		\end{overpic}  		%\includegraphics[width=0.75\textwidth]{stencil5points}
	\end{minipage}\hfill
	\begin{minipage}[b]
		{.49\textwidth}
		\centering
		\begin{overpic}[abs,width=0.65\textwidth,unit=1mm,scale=.25]{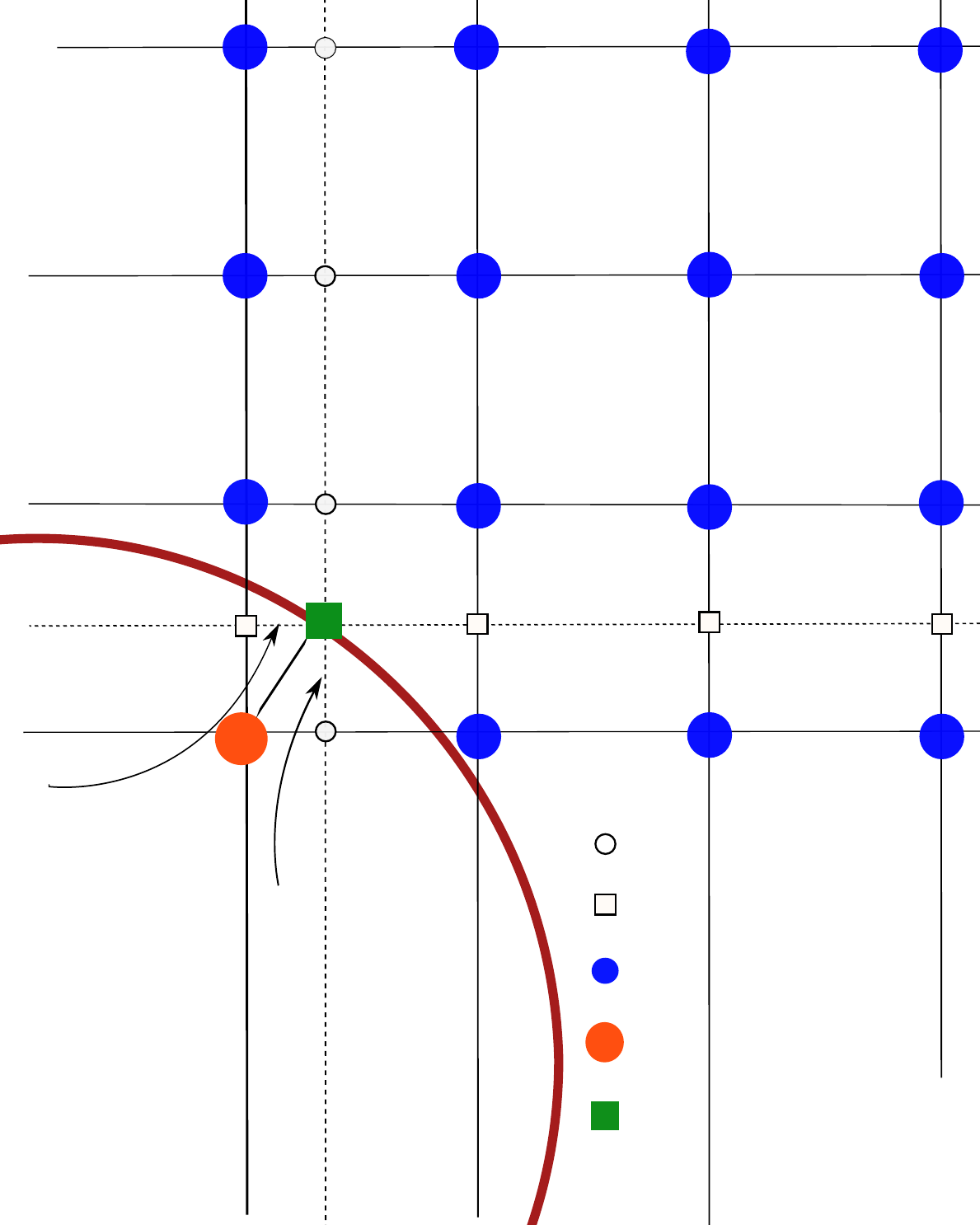}
			\put(10,12.5){\small $\vartheta_y h$}
			\put(0,17){\small $\vartheta_x h$}
			\put(-6.,45){(b)}
			\put(31,17.5){$x$--interpolation}
			\put(31,14.5){$y$--interpolation}
			\put(31,11.2){Inside points}
			\put(31,7){Ghost points}
			\put(31,4){Boundary points}
			\put(17,28.5){B}
			\put(7,18){G}
			\put(4,43){$\Omega$}
			\put(4,8){$\mathcal{B}$}
		\end{overpic}  
	\end{minipage}
	\hspace*{\fill}
	\caption{\textit{(a): 9-point stencil for the discrete operators $L_{{\rm i},h}^{\rm 4th}$ and $Q_{{\rm i},h}^{\rm 4th}$ for the internal points $P_{\rm i,j} = (x_i,y_j)$. (b): Representation of the upwind 16-points stencil associated with the ghost point $G$, boundary point $B$ and the relative outgoing normal vector $n$ to $\Gamma_\mathcal{B}$. The stencil is composed by 15 blue internal points, together with an orange ghost point G.}}
	\label{stencil}
\end{figure}
In this case, we rewrite Eq.~\eqref{eq_4th_space_variable} as follows
\begin{eqnarray}
	\label{eq_ghost_linear_system}
	\partial_t c_{\varepsilon,h} &=& \left(L^{\rm 4th}_{{\rm i},h} +  Q_{{\rm i},h}^{\rm 4th}(\vec{u}_h(t/\varepsilon)) \right)  c_{\varepsilon,h}, \quad c^0_{\varepsilon,h} = c_{\varepsilon,h}(t = 0) ,
\end{eqnarray}
where $L_{{\rm i},h}^{\rm 4th}$ and $Q_{{\rm i},h}^{\rm 4th}$ are $\mathcal{N} \times \mathcal{N}$ matrices representing the discretization of the derivative, including this time, the interpolation operators. 
%We denote by $L_h^{(i,j)}= \left(L_h^{(i,j),1},\ldots,L_h^{(i,j),N_I+N_G} \right)$ and $Q_h^{(i,j)}= \left(Q_h^{(i,j),1},\ldots,Q_h^{(i,j),N_I+N_G} \right)$ the rows of $L_h$ and $Q_h$, respectively, associated with the grid point $(x_i,y_j)$. 
If $P_{ij} = (x_i,y_j) \in \Omega_h$ is an internal grid point (as in Fig.~\ref{stencil} (a)), the expressions of $L_{{\rm i},h}^{\rm 4th}$ and $Q_{{\rm i},h}^{\rm 4th}$ coincide with ones of $L_{h}^{\rm 4th}$ and $Q_{h}^{\rm 4th}$ in Eqs.~\eqref{eq_L4th},\eqref{eq_Q4th}, respectively.  If $G=(x_i,y_j) \in \mathcal{G}_h$ is a ghost point, then we discretize the boundary condition in \eqref{eq:bc_ghost} for $\Gamma_{\mathcal{B}}$, following a ghost-point approach similar to the one proposed in \cite{COCO2013464,COCO2018299,ASTUTO2023111880,COCO202310}{, and} in \cite{cocohigh} they introduce a {higher} order discretization for the ghost points, and it is summarised as follows. We first compute the closest boundary point $B \in \Gamma_\mathcal{B}$ by
\[
B =O + R_\mathcal{B} \frac{O-G}{|O-G|},
\]
where $O$ is the center of the bubble. Then, we identify the upwind 16-point stencil starting from $G=(x_G,y_G)=(x_i,y_j)$, containing $B=(x_B,y_B)$:
\[
\left\{ (x_{i+s_x m_x},x_{j+s_y m_y}) \colon m_x,m_y=0,1,2,3 \right\},
\]
where $s_x = \text{sign} (x_B-x_G)$ and $s_y = \text{sign} (y_B-y_G)$. The solution $ c_{\varepsilon,h} $, along with its first and second derivatives, is interpolated at the boundary point $B$ using the discrete values $ c^\varepsilon_{i,j}$ on the 16-point stencil.
%%%---------------
% The interpolations can be obtained as tensor products of 1D interpolations in the axis directions. First, we interpolate in $r-$direction, moving from $G$ of a distance $\vartheta_x h$, and finding the three points in the rows $\{j, j+1, j+2\}$ (whole circle points in Fig.~\ref{stencil} (b)). Analogously, we move from $G$ in $y-$direction, of a distance $\vartheta_y h$, and we find the three points in the columns $\{i,i+1,i+2\}$. 
% \todo[inline]{add more details about interpolation}
% In detail, the 1D quadratic interpolations using the grid points $x_{i-2},x_{i-1},x_{i}$ to evaluate the function, its first derivative and the second derivative on $x_i - \vartheta h$ are given by
% \[
% \widehat{ c }(x_i + \vartheta\,h) =  \sum_{m=0}^2 \gamma_{m}(\vartheta) \,  c^\varepsilon_{i+m},
% \quad
% \widehat{ c }'(x_i + \vartheta\,h) =  \sum_{m=0}^2 \gamma'_{m}(\vartheta) \,  c^\varepsilon_{i+m},
% \quad
% \widehat{ c }''(x_i + \vartheta\,h) =  \sum_{m=0}^2 \gamma''_{m}(\vartheta)\, c^\varepsilon_{i+m},
% \]
% where
% \[
% \gamma(\vartheta) = \left( \frac{(1-\vartheta)(2-\vartheta)}{2}, \quad \vartheta (2-\vartheta), \quad \frac{\vartheta(\vartheta-1)}{2} \right)
% \]
% \[
% \gamma'(\vartheta) = \frac{1}{h} \left( \frac{(2\vartheta-3)}{2}, \quad 2(1-\vartheta), \quad \frac{(2\vartheta-1)}{2} \right)
% \]
% \[
% \gamma''(\vartheta) = \frac{1}{h^2} \left( 1, \quad -2, \quad 1 \right).
% \]
%---------------------------------
We start defining (see Fig.~\ref{stencil} (b))
\[ 
\vartheta_x =  s_x (x_B-x_G)/h, \qquad
\vartheta_y =  s_y (y_B-y_G)/h,
\]
with $0\leq \vartheta_x,\vartheta_y < 1$.
The 2D interpolation formulas are:
\begin{align*}
	\widehat{c}(B) &= \sum_{m_x,m_y=0}^3 l_{m_x}(\vartheta_x) l_{m_y}(\vartheta_y)  c _{i+s_x m_x,j+s_y m_y},
	\\
	\frac{\partial \widehat{c}}{\partial x}(B) &= s_x \sum_{m_x,m_y=0}^3 l'_{m_x}(\vartheta_x) l_{m_y}(\vartheta_y)  c _{i+s_x m_x,j+s_y m_y},
	%\\
	%\frac{\partial \widehat{c}}{\partial }(B) = s_y \sum_{m_x,m_y=0}^2 l_{m_x}(\vartheta_x) l'_{m_y}(\vartheta_y)  c _{i+s_x m_x,j+s_y m_y},
	\\
	\frac{\partial^2 \widehat{c}}{\partial x^2}(B) &= \sum_{m_x,m_y=0}^3 l''_{m_x}(\vartheta_x) l_{m_y}(\vartheta_y)  c _{i+s_x m_x,j+s_y m_y},
	%\\
	%\frac{\partial^2 \widehat{c}}{\partial y^2}(B) = \sum_{m_x,m_y=0}^2 l_{m_x}(\vartheta_x) l''_{m_y}(\vartheta_y)  c _{i+s_x m_x,j+s_y m_y},
	\\ 
	\frac{\partial^2 \widehat{c}}{\partial x \partial y}(B) &= s_x\, s_y \sum_{m_x,m_y=0}^3 l'_{m_x}(\vartheta_x) l'_{m_y}(\vartheta_y)  c _{i+s_x m_x,j+s_y m_y},
\end{align*}
where  
\begin{align*}
l(\vartheta_\alpha) & =  \left( \frac 1 2\left(1-\frac 1 2\vartheta_\alpha\right)\left(\frac 3 2\vartheta_\alpha-1\right)\left(\frac 3 2\vartheta_\alpha-2\right), \,\frac 9 4\vartheta_\alpha\left(\frac 3 2\vartheta_\alpha-2\right)\left(\frac 3 2\vartheta_\alpha-1\right), \right. \\  & \left. \, \frac 9 4\vartheta_\alpha\left(1-\frac 3 2\vartheta_\alpha\right)\left(\frac 1 2\vartheta_\alpha-1\right), \, \frac 1 4\vartheta_\alpha\left(\frac 3 2\vartheta_\alpha-1\right)\left(\frac 3 2\vartheta_\alpha-2\right) \right), \\
l'\left(\vartheta_\alpha\right) & = \frac{1}{h} \left( -\frac 1 6\left(\frac{27}{4}\vartheta_\alpha^2-18\vartheta_\alpha+11\right), \,\frac 3 2\left(\frac 9 4\vartheta_\alpha^2-5\vartheta_\alpha+2\right), \right. \\  & \left. \,-\frac 3 2\left(\frac 9 4\vartheta_\alpha^2-4\vartheta_\alpha+1\right), \,\frac 1 6\left(\frac{27}{4}\vartheta_\alpha^2-9\vartheta_\alpha+2\right) \right) \\
l''\left(\vartheta_\alpha\right) & = \frac{1}{h^2} \left( -2\left(\frac 1 2\vartheta_\alpha-1\right),3\vartheta_\alpha-5,-2\left(\frac 3 2\vartheta_\alpha-2\right),\vartheta_\alpha-1\right), \quad \alpha = x,y,
\end{align*}
%\giovanni{Secondo me manca un divisione per $h$ nella espressione della derivata prima ed una divisione per $h^2$ nella espressione della derivata seconda, poiche la derivata di $\vartheta$ rispetto a $x$ è $1/h$}
and where we omit ${\partial \widehat{c}}/{\partial y}(B)$ and ${\partial^2 \widehat{c}}/{\partial y^2}(B)$ because they are analogue to the $x-$coordinate derivatives. Finally, the rows of $L_{{\rm i},h}^{\rm 4th}$ corresponding to the ghost point $G=(x_G,y_G)$ are determined by applying the boundary condition on $\Gamma_\mathcal{B}$, i.e.
\begin{equation}\label{QHghost}
	L_{{\rm i},h}^{\rm 4th} c_{\varepsilon,h}\Big|_B = D \left. 
	\frac{\partial^2 \widehat{c}}{\partial \tau^2} 
	%\frac{\partial ^2 \widehat{ c }}{\partial \tau ^2} 
	\right|_B - 
	\frac{D}{\mathcal{M}} \left. \frac{\partial \widehat{ c }}{\partial n} \right|_B,
\end{equation}
and 
\begin{equation}
	\label{normaleq}
	\frac{\partial }{\partial \tau} = \tau_x\,\frac{\partial }{\partial x} + \tau_y\,\frac{\partial }{\partial y}, \quad
	\frac{\partial }{\partial n} = n_x\,\frac{\partial }{\partial x} + n_y\,\frac{\partial }{\partial y},
\end{equation}
% \begin{eqnarray}\label{normaleq}
% 	\frac{\partial }{\partial n} = n_x\,\frac{\partial }{\partial x} + n_y\,\frac{\partial }{\partial y}, &\qquad
% 	\displaystyle \frac{\partial ^2}{\partial \tau ^2} =
% 	\displaystyle \tau_x^2\,\frac{\partial^2 }{\partial x^2} + 2 \tau_x \tau_y\, \frac{\partial }{\partial x}\frac{\partial }{\partial y} + \tau_y^2\,\frac{\partial^2 }{\partial y^2}, %\\
% 	% \qquad
% 	%(n_x,n_y) = \frac{O-G}{|O-G|}, & \qquad (\tau_x,\tau_y)= (-n_y,n_x).
% \end{eqnarray}
with $(\tau_x,\tau_y)= (-n_y,n_x)$, $\cot \theta = n_x/n_y$. 
For a spherical bubble, \\
$(n_x,n_y) = ({O-G})/{|O-G|},\, \>\cot \theta =x/y$.
%At the end, the problem \eqref{system3D} is discretized in space, leading to a linear system
% \begin{equation}\label{eq_discr_space}
% {\partial_t} c_{\varepsilon,h} = \left(L_h + Q_h \vec{u}_h(t/\varepsilon)\right) c_{\varepsilon,h}, \qquad c_{\varepsilon,h}(t=0) = c^0_{\varepsilon,h}
% \end{equation}
% where $L_h$ and $Q_h$ are the $\mathcal{N} \times \mathcal{N}$ matrices representing the discretization of the diffusive Laplacian operator and $Q_h$ the advection term in system~\eqref{system3D}.

\subsection{Time discretization}
\label{section_time_discr}
In this section, we {develop} a scheme to solve Eq.~\eqref{eq_4th_space_variable}, that is uniformly accurate in $\varepsilon$. {The only assumption we make is} the uniform boundedness of the first time derivative of the solution with respect to $\varepsilon$. This assumption comes from the lack of oscillations in space of the initial condition { and time oscillations in the boundary conditions.}

As we mentioned in the Introduction, in this section, we start defining a numerical scheme that is first order accurate in time, and, recursively, we obtain second and third order numerical schemes. This paper is a natural continuation of a previous one \cite{astuto2023time}, where first and second order schemes are already defined and rigorously proved. Here, we recall the proof for the first order accurate numerical scheme to introduce the reader to the strategy used, that may seem overly technical.

\begin{pro} \label{pro_1st}
	Let ${L^{\rm 4th}_h}$ and ${Q^{\rm 4th}_h}$ be bounded operators in $\mathbb{R}^{\mathcal{N}}$ defined in Section~\ref{section_discr_space}, with $h>0$. Let $c_{\varepsilon,h}^0 \in \mathbb{R}^{\mathcal{N}}$ be bounded in $\varepsilon$, $\vec{u}_h \in \mathbb{R}^{\mathcal{N}}$ be a bounded given $1-$periodic function of time. Then, there exists a constant $\Delta t_0>0$ independent of $\varepsilon$ such that,  for all $\Delta t < \Delta t_0$, the following holds true:
	\begin{itemize}
		\item[i)] The operator $\mathcal{A}^1_{\Delta t} := I - \Delta t {L^{\rm 4th}_h} - {Q^{\rm 4th}_h}\int _{t^n}^{t^{n+1}}\vec{u}_h(s/\varepsilon)\,ds$ in $\mathbb{R}^\mathcal{N}$ is invertible.
		\item[ii)] The following:
		\begin{eqnarray*} 
			c_{\varepsilon,h}^0 &=& c_{\varepsilon,h}(0),  \\ 
			c_{\varepsilon,h}^{n+1} &=& {\mathcal{A}^1_{\Delta t}}^{-1}c_{\varepsilon,h}^n,
		\end{eqnarray*}
		for $n = 1,\cdots,M$ is a first order scheme, uniformly accurate with respect to $\varepsilon$, solving the Eq.~\eqref{eq_4th_space_variable}. In other words, we have $||c_{\varepsilon,h}(t^n) - c_{\varepsilon,h}^n||\leq K \Delta t$ for all $n = 1,\cdots,M$, with $K$ independent of $\varepsilon$, $\Delta t = t_{\rm fin}/M$, $t_{\rm fin}> 0$ and $t^n = n\Delta t$.    
	\end{itemize}
	The same result is valid also for the Eq.~\eqref{eq_ghost_linear_system}, where the operators are $L^{\rm 4th}_{{\rm i},h}$ and $Q^{\rm 4th}_{{\rm i},h}$, and the proof is analogue.
\end{pro}

{We omit the proof of the first-order scheme, as it is available in \cite[Proposition 1]{astuto2023time}.}

\begin{pro} \label{pro_2nd}
	Let ${L^{\rm 4th}_h}$ and ${Q^{\rm 4th}_h}$ be finite dimensional bounded operators in $\mathbb{R}^{\mathcal{N}}$ defined in Section~\ref{section_discr_space}, with $h>0$. Let $c_{\varepsilon,h}^0 \in \mathbb{R}^{\mathcal{N}}$ be bounded in $\varepsilon$, and let $\vec{u}_h \in \mathbb{R}^{\mathcal{N}}$ be a bounded given $1-$periodic function in time. Then, there exists a constant $\Delta t_0>0$ independent of $\varepsilon$, such that for all $\Delta t < \Delta t_0$, the following holds true:
	\begin{itemize}
		\item[i)] The operator $\mathcal{A}^2_{\Delta t} = \mathcal{A}^1_{\Delta t} - \mathbb{M}_2$ in $\mathbb{R}^{\mathcal{N}}$ is invertible, where
		\begin{align} \nonumber & \mathbb{M}_2 = - \frac{1}{2}\left(L^{\rm 4th}_h\right)^2\Delta t^2 - {L^{\rm 4th}_h}{Q^{\rm 4th}_h}\int  _{t^n}^{t^{n+1}} \int_{s}^{t^{n+1}}\vec{u}_h(\sigma/\varepsilon)\,d\sigma ds \\ \nonumber &   - {Q^{\rm 4th}_h}{L^{\rm 4th}_h}\int_{t^n}^{t^{n+1}}(t^{n+1} - s)\vec{u}_h(s/\varepsilon)\,ds \\ \nonumber & - \left(Q^{\rm 4th}_h\right)^2 \int_{t^n}^{t^{n+1}}\vec{u}_h(s/\varepsilon)\int_{s}^{t^{n+1}}\vec{u}_h(\sigma/\varepsilon)\,d\sigma\,ds, 
		\end{align}    
		\item[ii)] The following scheme:
  \begin{subequations}
\label{eq_scheme_2nd}		
\begin{eqnarray} \label{eq_scheme_2nd_1}
			c_{\varepsilon,h}^0 &=& c_{\varepsilon,h}(0),  \\ \label{eq_scheme_2nd_2}
			c_{\varepsilon,h}^{n+1} &=& {\mathcal{A}^2_{\Delta t}}^{-1}c_{\varepsilon,h}^n, \qquad {\rm for } \quad n = 1,\cdots,M
		\end{eqnarray}
  \end{subequations}
		is a second order scheme, uniformly accurate with respect to $\varepsilon$, solving the Eq.~\eqref{eq_4th_space_variable}. In other words we have $||c_h(t^n) - c_h^n||\leq K \Delta t^2$,  for all $n = 1,\cdots,M$, with $K$ independent of $\varepsilon$, $t^n = n\Delta t$ and $\Delta t = t_{\rm fin}/M$.  
	\end{itemize}
\end{pro}
{We skip the proof of the second order scheme because it can be found in \cite[Proposition 2]{astuto2023time}.}

\begin{pro} \label{pro_3rd}
	Let ${L^{\rm 4th}_h}$ and ${Q^{\rm 4th}_h}$ be finite dimensional bounded operators in $\mathbb{R}^{\mathcal{N}}$ defined in Section~\ref{section_discr_space}, with $h>0$. Let $c_{\varepsilon,h}^0 \in \mathbb{R}^{\mathcal{N}}$ be bounded in $\varepsilon$, and let $\vec{u}_h \in \mathbb{R}^{\mathcal{N}}$ be a bounded given $1-$periodic function in time. Then, there exists a constant $\Delta t_0>0$ independent of $\varepsilon$, such that for all $\Delta t < \Delta t_0$, the following holds true:
	\begin{itemize}
		\item[i)] The operator $\mathcal{A}^3_{\Delta t} = \mathcal{A}^2_{\Delta t} - \mathbb{M}_3$ in $\mathbb{R}^{\mathcal{N}}$ is invertible, where
		\begin{align} \nonumber \mathbb{M}_3 & = \frac{1}{2}\left(L^{\rm 4th}_h\right)^3\int_{t^n}^{t^{n+1}}(t^{n+1}-s)^2\,ds \\ \nonumber &  + \left(L^{\rm 4th}_h\right)^2{Q^{\rm 4th}_h} \int_{t^n}^{t^{n+1}}\int_{s}^{t^{n+1}}\int_{\sigma}^{t^{n+1}}\vec{u}_h(\rho/\varepsilon)d\rho\, d\sigma
			\\ \nonumber & + {L^{\rm 4th}_h} {Q^{\rm 4th}_h} {L^{\rm 4th}_h} \int_{t^n}^{t^{n+1}}\int_{s}^{t^{n+1}} (t^{n+1}-\sigma)\vec{u}_h(\sigma/\varepsilon) d\sigma ds 
			\\\nonumber & + {L^{\rm 4th}_h} \left(Q^{\rm 4th}_h\right)^2 \int_{t^n}^{t^{n+1}} \int_{s}^{t^{n+1}} \vec{u}_h(\sigma/\varepsilon) \int_{\sigma}^{t^{n+1}}\vec{u}_h(\rho/\varepsilon) d\rho\, d\sigma\,ds   \\ \nonumber & + \frac{1}{2} {Q^{\rm 4th}_h} \left(L^{\rm 4th}_h\right)^2 \int_{t^n}^{t^{n+1}} (t^{n+1}-s)^2\vec{u}_h(s/\varepsilon)ds
			\\ \nonumber &  + {Q^{\rm 4th}_h} {L^{\rm 4th}_h} {Q^{\rm 4th}_h} \int_{t^n}^{t^{n+1}} \vec{u}_h(s/\varepsilon)\int_{s}^{t^{n+1}}\int_{\sigma}^{t^{n+1}}\vec{u}_h(\rho/\varepsilon) d\rho\, d\sigma\,ds  \\ \nonumber & + \left(Q^{\rm 4th}_h\right)^2 {L^{\rm 4th}_h} \int_{t^n}^{t^{n+1}}\vec{u}_h(s/\varepsilon) \int_{s}^{t^{n+1}}(t^{n+1}-\sigma )\vec{u}_h(\sigma/\varepsilon) d\sigma \,ds  \\ \nonumber &
			+ \left(Q^{\rm 4th}_h\right)^3 \int_{t^n}^{t^{n+1}} \vec{u}_h(s/\varepsilon) \int_{s}^{t^{n+1}} \vec{u}_h(\sigma/\varepsilon) \int_{\sigma}^{t^{n+1}} \vec{u}_h(\rho/\varepsilon) d\rho\, d\sigma\,ds
		\end{align}    
		\item[ii)] The following scheme:
  \begin{subequations}
  \label{eq_scheme_3rd}
		\begin{eqnarray} \label{eq_scheme_3rd_1}
			c_{\varepsilon,h}^0 &=& c_{\varepsilon,h}(0),  \\ \label{eq_scheme_3rd_2}
			c_{\varepsilon,h}^{n+1} &=& {\mathcal{A}^3_{\Delta t}}^{-1}c_{\varepsilon,h}^n, \qquad {\rm for } \quad n = 1,\cdots,M
		\end{eqnarray}
\end{subequations}
		is a third order scheme, uniformly accurate with respect to $\varepsilon$, solving the Eq.~\eqref{eq_4th_space_variable}. In other words, we have $||c_h(t^n) - c_h^n||\leq K \Delta t^3$,  for all $n = 1,\cdots,M$, with $K$ independent of $\varepsilon$, $t^n = n\Delta t$ and $\Delta t = t_{\rm fin}/M$.  
	\end{itemize}
	The same result is valid also for the Eq.~\eqref{eq_ghost_linear_system}, where the operators are $L^{\rm 4th}_{{\rm i},h}$ and $Q^{\rm 4th}_{{\rm i},h}$, and the proof is analogue. {We will follow the iterative strategy introduced in \cite[Propositions 1,2]{astuto2023time}.}
\end{pro}
\begin{proof}
	To prove i), we prove that $\mathcal{A}^3_{\Delta t}$ is invertible when $\Delta t \to 0$. We start considering the norm of the operator $\mathcal{A}^3_{\Delta t}${. Since} the operators ${L^{\rm 4th}_h}$ and ${Q^{\rm 4th}_h}$ are bounded in $\mathbb{R}^{\mathcal{N}}$, and the vector-valued function $\vec{u}_h$ is also bounded, it follows that
	$||\mathcal{A}^3_{\Delta t}||\to ||I||$ when $\Delta t \to 0$. {Thus, we say} that there exists a constant $\widehat K > 0$, such that, $\forall \Delta t < \widehat K$, the operator $\mathcal{A}^3_{\Delta t}$ is invertible.
	
	To prove ii), we first show how to deduce the scheme defined in {Eqs.~\eqref{eq_scheme_3rd}}. {We} start integrating  Eq.~\eqref{eq_4th_space_variable} between $t$ and $t^{n+1}$, with $[t,t^{n+1}] \subset [t^n, t_{\rm fin}]$ 
	\begin{eqnarray}
		c_{\varepsilon,h}(t^{n+1}) -  c_{\varepsilon,h}(t) &=& \int_{t}^{t^{n+1}}\left({L^{\rm 4th}_h} + {Q^{\rm 4th}_h} \vec{u}_h(s/\varepsilon)\right) c_{\varepsilon,h}(s) ds.
		\label{eq_integral_form_3rd}
	\end{eqnarray} 
	that can be rewritten as
	\begin{eqnarray}
		\label{eq_integral_form_s}
		c_{\varepsilon,h}(s) = c_{\varepsilon,h}(t^{n+1}) - 
		\int_{s}^{t^{n+1}}\left({L^{\rm 4th}_h} + {Q^{\rm 4th}_h} \vec{u}_h(\sigma/\varepsilon)\right) c_{\varepsilon,h}(\sigma) d\sigma.
	\end{eqnarray}
	%where we remind that the choice of integrating between $t$ and $t^{n+1}$, with $t<t^{n+1}$, is our strategy to obtain a negative sign in front of the operator $\left(L^{\rm 4th}_h\right)^2$ in Eq.~\eqref{eq_2nd_order_method} 
	%\todo[inline]{write few words} (see Remark for more details).
	
	Now we substitute in Eq.~\eqref{eq_integral_form_3rd} the expression for $c_{\varepsilon,h}(s)$ in Eq.~\eqref{eq_integral_form_s}, as follows
	\begin{eqnarray} \nonumber
		c_{\varepsilon,h}(t^{n+1}) -  c_{\varepsilon,h}(t) &=& \int_{t}^{t^{n+1}}\left({L^{\rm 4th}_h} + {Q^{\rm 4th}_h} \vec{u}_h(s/\varepsilon)\right) c_{\varepsilon,h}(s) ds\\ \label{3rd_order_def_sigma}
		&=& \int_{t}^{t^{n+1}}\left({L^{\rm 4th}_h} + {Q^{\rm 4th}_h} \vec{u}_h(s/\varepsilon)\right)\,ds\, c_{\varepsilon,h}(t^{n+1})  \\ 
		&& - \int_{t}^{t^{n+1}}\left({L^{\rm 4th}_h} + {Q^{\rm 4th}_h} \vec{u}_h(s/\varepsilon)\right)\int_{s}^{t^{n+1}}\left({L^{\rm 4th}_h} + {Q^{\rm 4th}_h} \vec{u}_h(\sigma/\varepsilon)\right) c_{\varepsilon,h}(\sigma)d\sigma\,ds. \nonumber
	\end{eqnarray} 
	Since the high order in this numerical scheme is achieved recursively, we write the same expression in Eq.~\eqref{eq_integral_form_s} for $c_{\varepsilon,h}(\sigma)$,
	\begin{eqnarray}
		\label{3rd_order_def}
		c_{\varepsilon,h}(\sigma) = c_{\varepsilon,h}(t^{n+1}) - 
		\int_{\sigma}^{t^{n+1}}\left({L^{\rm 4th}_h} + {Q^{\rm 4th}_h} \vec{u}_h(\rho/\varepsilon)\right) c_{\varepsilon,h}(\rho) d\rho,
	\end{eqnarray}
	and, analogously, we substitute in Eq.~\eqref{3rd_order_def_sigma} the expression of $c_{\varepsilon,h}(\sigma)$ in Eq.~\eqref{3rd_order_def}, obtaining
	\begin{eqnarray} \label{3rd_order_def_sigma_2}
		c_{\varepsilon,h}(t^{n+1}) -  c_{\varepsilon,h}(t) %&& = %\int_{t}^{t^{n+1}}\left({L^{\rm 4th}_h} + {Q^{\rm 4th}_h} \vec{u}_h(s/\varepsilon)\right) c_{\varepsilon,h}(s) ds
		&& = \int_{t}^{t^{n+1}}\left({L^{\rm 4th}_h} + {Q^{\rm 4th}_h} \vec{u}_h(s/\varepsilon)\right) \,ds\, c_{\varepsilon,h}(t^{n+1})\\ && \nonumber - \int_{t}^{t^{n+1}}\left({L^{\rm 4th}_h} + {Q^{\rm 4th}_h} \vec{u}_h(s/\varepsilon)\right)\int_{s}^{t^{n+1}}\left({L^{\rm 4th}_h} + {Q^{\rm 4th}_h} \vec{u}_h(\sigma/\varepsilon)\right)d\sigma\,ds\,c_{\varepsilon,h}(t^{n+1}) \\ && \nonumber + \int_{t}^{t^{n+1}}\left({L^{\rm 4th}_h} + {Q^{\rm 4th}_h} \vec{u}_h(s/\varepsilon)\right)\int_{s}^{t^{n+1}}\left({L^{\rm 4th}_h} + {Q^{\rm 4th}_h} \vec{u}_h(\sigma/\varepsilon)\right) \\ && \nonumber \cdot\int_{\sigma}^{t^{n+1}}\left({L^{\rm 4th}_h} + {Q^{\rm 4th}_h} \vec{u}_h(\rho/\varepsilon)\right) c_{\varepsilon,h}(\rho)d\rho\,d\sigma\,ds.
		\nonumber 
	\end{eqnarray} 
	
	{Now}, we evaluate the quantity $ c_{\varepsilon,h}(t)$ in $t=t^n$, {such that $\Delta t = t^{n+1} - t^n < \widehat K$,} and approximate the quantities in the following way
	\begin{eqnarray}  \label{3rd_order_appr_sigma}
		 c_{\varepsilon,h}^{n+1} -  c_{\varepsilon,h}^n = && \int_{t^n}^{t^{n+1}}\left({L^{\rm 4th}_h} + {Q^{\rm 4th}_h} \vec{u}_h(s/\varepsilon)\right) \,ds\, c_{\varepsilon,h}^{n+1} \\ &&  \notag - \int_{t^n}^{t^{n+1}}\left({L^{\rm 4th}_h} + {Q^{\rm 4th}_h} \vec{u}_h(s/\varepsilon)\right)\int_{s}^{t^{n+1}}\left({L^{\rm 4th}_h} + {Q^{\rm 4th}_h} \vec{u}_h(\sigma/\varepsilon)\right) d\sigma\,ds\,c_{\varepsilon,h}^{n+1} \\ \nonumber  && + \int_{t^n}^{t^{n+1}}\left({L^{\rm 4th}_h} + {Q^{\rm 4th}_h} \vec{u}_h(s/\varepsilon)\right)\int_{s}^{t^{n+1}}\left({L^{\rm 4th}_h} + {Q^{\rm 4th}_h} \vec{u}_h(\sigma/\varepsilon)\right) \\ \nonumber && \cdot\int_{\sigma}^{t^{n+1}}\left({L^{\rm 4th}_h} + {Q^{\rm 4th}_h} \vec{u}_h(\rho/\varepsilon)\right) d\rho\,d\sigma\,ds\, c_{\varepsilon,h}^{n+1}, \nonumber
	\end{eqnarray} 
	where $c^n_{\varepsilon,h} \approx c_{\varepsilon,h}(t^n)$.

	At this point, we show that the numerical scheme is third order accurate in time, uniformly in $\varepsilon$. {We} subtract Eq.~\eqref{3rd_order_appr_sigma} from Eq.~\eqref{3rd_order_def_sigma_2}{, obtaining} 
	\begin{eqnarray} \nonumber
	 e_{n+1} - e_n = && \int_{t^n}^{t^{n+1}}\left({L^{\rm 4th}_h} + {Q^{\rm 4th}_h} \vec{u}_h(s/\varepsilon)\right) e_{n+1}\,ds \\  \nonumber && - \int_{t^n}^{t^{n+1}}\left({L^{\rm 4th}_h} + {Q^{\rm 4th}_h} \vec{u}_h(s/\varepsilon)\right)\int_{s}^{t^{n+1}}\left({L^{\rm 4th}_h} + {Q^{\rm 4th}_h} \vec{u}_h(\sigma/\varepsilon)\right) e_{n+1}\,d\sigma\,ds \\ \nonumber && + \int_{t^n}^{t^{n+1}}\left({L^{\rm 4th}_h} + {Q^{\rm 4th}_h} \vec{u}_h(s/\varepsilon)\right)\int_{s}^{t^{n+1}}\left({L^{\rm 4th}_h} + {Q^{\rm 4th}_h} \vec{u}_h(\sigma/\varepsilon)\right) \\ \label{eq_en_4th_order} &&\cdot \int_{\sigma}^{t^{n+1}}\left({L^{\rm 4th}_h} + {Q^{\rm 4th}_h} \vec{u}_h(\rho/\varepsilon)\right) (c_{\varepsilon,h}(\rho) - c_{\varepsilon,h}^{n+1})d\rho d\sigma ds.
	\end{eqnarray} 
	{where} $e_n = c_{\varepsilon,h}(t^n) - c_{\varepsilon,h}^n$. We add and subtract the same quantity to the right hand side, saying that, $\, c_{\varepsilon,h}(\rho) - c_{\varepsilon,h}^{n+1} = c_{\varepsilon,h}(\rho) -c_{\varepsilon,h}(t^{n+1})+c_{\varepsilon,h}(t^{n+1})- c_{\varepsilon,h}^{n+1}, \,$ {and Eq.~\eqref{eq_en_4th_order} becomes
 \begin{subequations}
 \begin{eqnarray} \nonumber
	 e_{n+1} - e_n = && \int_{t^n}^{t^{n+1}}\left({L^{\rm 4th}_h} + {Q^{\rm 4th}_h} \vec{u}_h(s/\varepsilon)\right) \,ds\,e_{n+1} \\  \nonumber && - \int_{t^n}^{t^{n+1}}\left({L^{\rm 4th}_h} + {Q^{\rm 4th}_h} \vec{u}_h(s/\varepsilon)\right)\int_{s}^{t^{n+1}}\left({L^{\rm 4th}_h} + {Q^{\rm 4th}_h} \vec{u}_h(\sigma/\varepsilon)\right) \,d\sigma\,ds\,e_{n+1} \\ \label{eq_en_4th_orderA} && + \int_{t^n}^{t^{n+1}}\left({L^{\rm 4th}_h} + {Q^{\rm 4th}_h} \vec{u}_h(s/\varepsilon)\right)\int_{s}^{t^{n+1}}\left({L^{\rm 4th}_h} + {Q^{\rm 4th}_h} \vec{u}_h(\sigma/\varepsilon)\right) \\ \label{eq_en_4th_orderB} &&\cdot \int_{\sigma}^{t^{n+1}}\left({L^{\rm 4th}_h} + {Q^{\rm 4th}_h} \vec{u}_h(\rho/\varepsilon)\right) (c_{\varepsilon,h}(\rho) - c_{\varepsilon,h}(t^{n+1}))d\rho d\sigma ds \\ \nonumber && + \int_{t^n}^{t^{n+1}}\left({L^{\rm 4th}_h} + {Q^{\rm 4th}_h} \vec{u}_h(s/\varepsilon)\right)\int_{s}^{t^{n+1}}\left({L^{\rm 4th}_h} + {Q^{\rm 4th}_h} \vec{u}_h(\sigma/\varepsilon)\right) \\ \nonumber &&\cdot \int_{\sigma}^{t^{n+1}}\left({L^{\rm 4th}_h} + {Q^{\rm 4th}_h} \vec{u}_h(\rho/\varepsilon)\right) d\rho d\sigma ds\,e_{n+1}. 
	\end{eqnarray} 
  \label{eq_en_4th_order2}
 \end{subequations}
 Given the boundedness of $\vec{u}_h(t/\varepsilon)$ and $c_{\varepsilon,h}(t)$, we deduce that the first derivative in time of $c_{\varepsilon,h}(t)$ is also bounded. Thus, there exists a constant $K_1$, independent of $\varepsilon$, such that
	\[ ||\partial_t  c_{\varepsilon,h} ||_{\mathcal{L}^\infty(\Omega)} = \left|\left|\left({L^{\rm 4th}_h} + {Q^{\rm 4th}_h} \vec{u}_h(t/\varepsilon)\right) c _{\varepsilon,h}\right|\right|_{\mathcal{L}^\infty(\Omega)} < K_1.
	\]
	where $K_1 = \left(||{L^{\rm 4th}_h}|| + ||{Q^{\rm 4th}_h}||\,||u||_\infty\right) ||c_{\varepsilon,h}||_{\mathcal{L}^\infty(\Omega)}$. From this estimate, we obtain 
	\begin{equation}
		||c_{\varepsilon,h}(s) - c_{\varepsilon,h}(t^{n+1})||\leq K_1 ||s - t^{n+1}|| \leq K_1 \Delta t, \qquad \forall s \in [t^n,t^{n+1}].
	\end{equation}
Looking at the term in Eqs.~(\ref{eq_en_4th_orderA}-\ref{eq_en_4th_orderB}), it is bounded by
\begin{eqnarray*}
   && \left|\left| \int_{t^n}^{t^{n+1}}\left({L^{\rm 4th}_h} + {Q^{\rm 4th}_h} \vec{u}_h(s/\varepsilon)\right)\int_{s}^{t^{n+1}}\left({L^{\rm 4th}_h} + {Q^{\rm 4th}_h} \vec{u}_h(\sigma/\varepsilon)\right) \right.\right. \\ && \left.\left.\cdot \int_{\sigma}^{t^{n+1}}\left({L^{\rm 4th}_h} + {Q^{\rm 4th}_h} \vec{u}_h(\rho/\varepsilon)\right) (c_{\varepsilon,h}(\rho) - c_{\varepsilon,h}(t^{n+1}))d\rho d\sigma ds \right|\right|  \leq K_1 K_2 \Delta t^2 
\end{eqnarray*}
where $K_2 = ||{L^{\rm 4th}_h} + {Q^{\rm 4th}_h} \vec{u}_h(s/\varepsilon)||_{\mathcal{L}^\infty(\Omega)}$, and, it does not depend on $\varepsilon$. Considering now the norm of Eq.~\eqref{eq_en_4th_order2}, the following inequality holds
%and, using the analogue procedure and the same constants seen in Eq.~\eqref, Eq.~\eqref{eq_en_4th_order} becomes
	\begin{equation*}
		\left( 1 - K_3\Delta t - K_3^2\Delta t^2 - K_3^3\Delta t^3\right) E_{n+1} - K_3^3\Delta t^4 \leq E_n  
	\end{equation*}
	where $E_n = ||e_n||$ and $K_3 = \max\{K_1K_2,K_2\}$. }
	%%%%%%%%%%%%%%%%%%%%%%%%%%%%%%%%%%%%%%
	%%%%%%%%%%%%%%%%%%%%%%%%%%%%%%%%%%%%%%%
	{For the third order accurate numerical scheme}, we add ${K_3^2 \Delta t^3}/({1+K_3\Delta t + K_3^2\Delta t^2})$ to both sides, and, after some algebra, the inequality reads
	\begin{align}
	& E_{n+1} + \frac{K_3^2 \Delta t^3}{1+K_3\Delta t + K_3^2\Delta t^2} \leq  \\ & \left( 1 - K_3\Delta t - K_3^2\Delta t^2 - K_3^3\Delta t^3\right)^{-1} \left(E_{n} +\frac{K_3^2 \Delta t^3}{1+K_3\Delta t + K_3^2\Delta t^2}\right),
	\end{align}
	if $ 1 - K_3\Delta t - K_3^2\Delta t^2 - K_3^3\Delta t^3 > 0$. Recursively, we obtain (using $E_0 = 0$) 
	\[
	E_{n+1} + \frac{K_3^2 \Delta t^3}{1+K_3\Delta t + K_3^2\Delta t^2} \leq \left( 1 - K_3\Delta t - K_3^2\Delta t^2 - K_3^3\Delta t^3\right)^{-n} \frac{K_3^2 \Delta t^3}{1+K_3\Delta t + K_3^2\Delta t^2}.
	\]
	{It} is easy to show that there exists a constant {$ K_4$} that does not depend on $N$, such that {$E_n \leq K_4 \Delta t^3$}. %To conclude the proof, we define $\Delta t_0 = \min\{\widehat K, \widetilde K\}$, where $\widetilde K$ is the positive solution of the equation $1 - K_3\Delta t - K_3^2\Delta t^2 - K_3^3\Delta t^3 = 0$.
 %%%%%%%%%%%%%%%%%%%%%%
 To conclude the proof, {by induction, we can verify that $||c_{\varepsilon,h}(t^n) - c_{\varepsilon,h}^n||\leq K \Delta t^3$,  for all $n = 1,\cdots,M$, provided $\Delta t < \Delta t_0$}, with $\Delta t_0 = \min\{\widehat K, \widetilde K\}$, where $\widetilde K$ is the positive solution of the equation $1 - K_3\Delta t - K_3^2\Delta t^2 - K_3^3\Delta t^3 = 0$.
\end{proof}

\section{Numerical results}
\label{section_results}
In this section, we show the accuracy tests of the different numerical schemes that we derived in the previous sections. The domain considered in the following tests is $\Omega = \left( [-1,1]\times[-1,1]\right)$. We start with the fourth order accurate spatial numerical scheme for the drift-diffusion equation in Eq.~\eqref{eq_4th_space}, and for the one with variable coefficients for the drift term in Eq.~\eqref{eq_4th_space_variable}. 
\begin{figure}[!ht]
	\centering
	\begin{minipage}[b]
		{.49\textwidth}
		\centering
		\begin{overpic}[abs,width=\textwidth,unit=1mm,scale=.25]{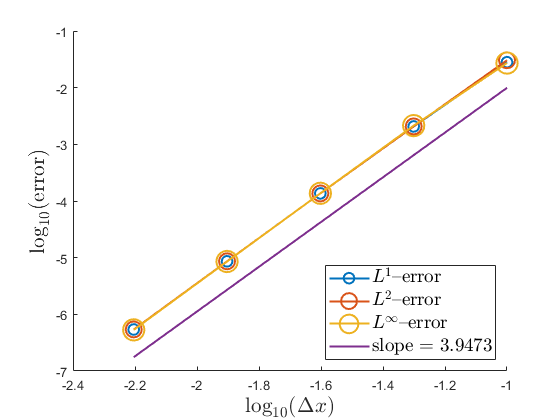}
  \put(0.,53){(a)}
            \end{overpic}
            \end{minipage}
	\begin{minipage}[b]
		{.49\textwidth}
		\centering
		\begin{overpic}[abs,width=\textwidth,unit=1mm,scale=.25]{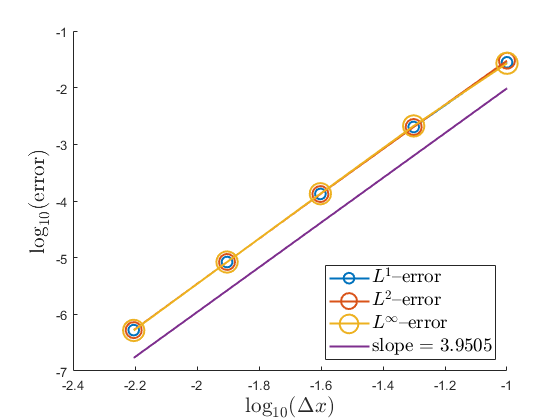}
  \put(0.,53){(b)}
  \end{overpic}
  \end{minipage}
	\caption{\textit{Relative error in $L^1,L^2,L^\infty$--norms, with $t_{\rm fin } = 0.1,$ for a fixed $\Delta t_{\rm ref} = 10^{-5}$ for Eq.~\eqref{eq_4th_space} (a) and Eq.~\eqref{eq_4th_space_variable} (b). The domain is $\Omega = [-1,1]^2$,  $N_{\rm ref} = 640$, the initial condition is defined in Eq.~\eqref{eq_expr_IC_tests}  with $x_{m_1} = y_{m_1} = 0, \sigma = 0.1$ homogeneous Dirichlet boundary conditions (i.e., $f = 0$ in Eq.~\eqref{eq_Dir_bc}) and velocity in Eq.~\eqref{eq_expr_velocity_space} for (b).}}
	\label{fig_4th_space_variable}
\end{figure}
\begin{figure}[H]
	\centering
	\centering
	\hfill
	\begin{minipage}[b]
		{.49\textwidth}
		\centering
		\begin{overpic}[abs,width=\textwidth,unit=1mm,scale=.25]{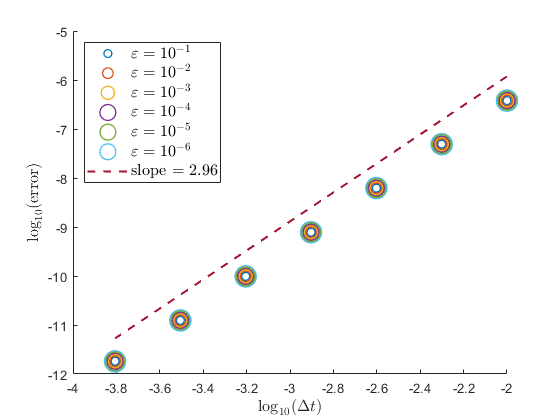}
			\put(0.,44){(a)}
		\end{overpic}	
	\end{minipage}\hfill
	\begin{minipage}[b]
		{.49\textwidth}
		\centering
		\begin{overpic}[abs,width=\textwidth,unit=1mm,scale=.25]{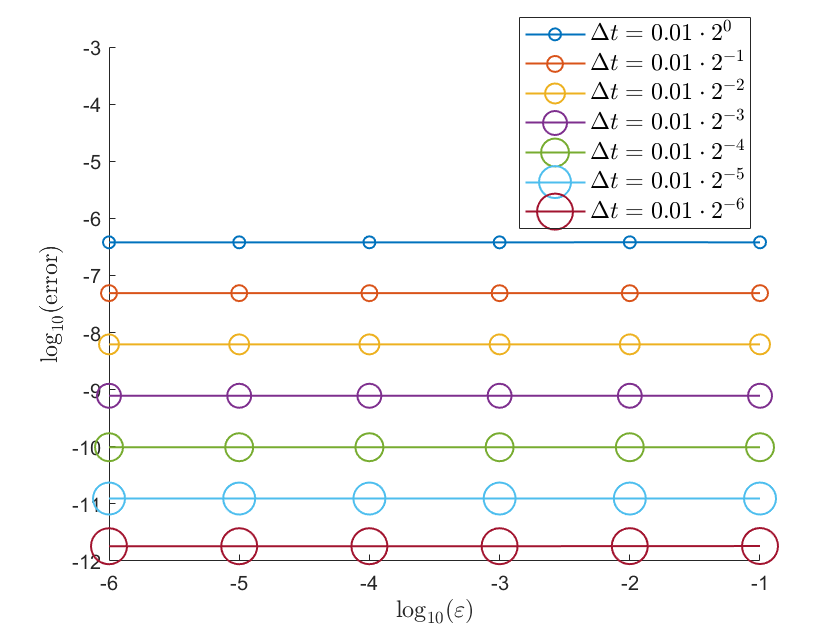}
			\put(0.,44){(b)}
		\end{overpic}	
	\end{minipage}
	\caption{\textit{Plot of the relative  error in $L^2$--norm for the third order scheme in {Eqs.~\eqref{eq_scheme_3rd}} as a function of $\Delta t$ and different $\varepsilon$ (a) and as a function of $\varepsilon$ and different $\Delta t$ (b). The parameters of the tests are: $t_{\rm fin}=0.1, \Delta t_{\rm ref} = 10^{-5},  N_{\rm ref} = 160, D = 0.02, \delta = 10^{-2}$. The domain is $\Omega = [-1,1]^2$, the initial condition is defined in Eq.~\eqref{eq_expr_IC_tests}  with $x_{m_1} = y_{m_1} = 0$ and $\sigma = 0.1$, and the velocity expression in Eq.~\eqref{eq_expr_velocity}, $A = 1$.}}
	\label{fig_time_3rd_order}
\end{figure}
To conclude the accuracy tests of the space discretization, we test the third order accuracy of the numerical scheme defined in Eq.~\eqref{eq_ghost_linear_system}. In this last case, the domain considered is $\Omega = \left( [-1,1]\times[-1,1]\right)\setminus \mathcal{B}$, where $\mathcal{B}$ is a circle centered in $(0,0)$ with radius $R_\mathcal{B} = 0.2$. For the squared external boundaries, we choose homogeneous Dirichlet boundary conditions, i.e., $f = 0$ in Eq.~\eqref{eq_Dir_bc}. Subsequently, we present numerical tests for the third-order accurate time discretization. We focus on the uniformly accuracy in $\varepsilon$ of the numerical method, and after that, we show the advantages of the arbitrariness of choosing the time step $\Delta t$, without losing accuracy.
\begin{table}[H] 
	\centering
	\begin{tabular}{||c||c||c||c||c||}\hline \hline
		N   & ${\rm error}_1$ & ${\rm error}_2$ & ${\rm error}_\infty$ & order \\ \hline \hline
		10 & 2.815$\cdot 10^{-2}$ &  3.049$\cdot 10^{-2}$ & 2.718$\cdot 10^{-2}$ & -- \\ \hline \hline
		20 & 2.069$\cdot 10^{-3}$ &  2.073$\cdot 10^{-3}$ &  2.140$\cdot 10^{-3}$ & 3.77\\ \hline \hline
		40 & 1.363$\cdot 10^{-4}$ & 1.362$\cdot 10^{-4}$ &1.367$\cdot 10^{-4}$ & 3.92
		\\ \hline \hline
		80 & 8.630$\cdot 10^{-6}$ & 8.630$\cdot 10^{-6}$ & 8.589$\cdot 10^{-6}$ & 3.98\\ \hline \hline 
		160 & 5.407$\cdot 10^{-7}$ & 5.409$\cdot 10^{-7}$ &
		5.320$\cdot 10^{-7}$ & 4 \\ \hline \hline
	\end{tabular}
	\caption{\textit{Convergence rate of the spatial numerical discretization defined in Eq.~\eqref{eq_4th_space}. The relative error is calculated in $L^1,L^2,L^\infty$--norms, for a fixed $\Delta t_{\rm ref} = 10^{-5}$ and $t_{\rm fin} = 0.1$. The domain is $\Omega = [-1,1]^2$, the initial condition is defined in Eq.~\eqref{eq_expr_IC_tests}  with $x_{m_1} = y_{m_1} = 0, \sigma = 0.1$ and homogeneous Dirichlet boundary conditions (i.e., $f = 0$ in Eq.~\eqref{eq_Dir_bc}).}}
	\label{table_4th_space}
\end{table}
%%%%%%%%%
At the end, we define three different level set functions for arbitrary domains. The problem we consider is the advection-diffusion equation coupled by homogeneous Dirichlet boundary conditions, in system \eqref{eq:Dirichlet_arbitrary}. For these tests, we show fourth order space accuracy.
%%%%%%
\begin{table}[H]
	\centering
	\begin{tabular}{||c||c||c||c||c||}\hline \hline
		N   & ${\rm error}_1$ & ${\rm error}_2$ & ${\rm error}_\infty$ & order \\ \hline \hline
		10 & 2.800$\cdot 10^{-2}$ &  2.993$\cdot 10^{-2}$ & 2.697$\cdot 10^{-2}$ & -- \\ \hline \hline
		20 & 2.024$\cdot 10^{-3}$ &  2.024$\cdot 10^{-3}$ &  2.112$\cdot 10^{-3}$ & 3.79\\ \hline \hline
		40 & 1.332$\cdot 10^{-4}$ & 1.326$\cdot 10^{-4}$ &1.345$\cdot 10^{-4}$ & 3.93
		\\ \hline \hline
		80 & 8.394$\cdot 10^{-6}$ & 8.393$\cdot 10^{-6}$ & 8.378$\cdot 10^{-6}$ & 3.99\\ \hline \hline 
		160 & 5.199$\cdot 10^{-7}$ & 5.261$\cdot 10^{-7}$ &
		5.320$\cdot 10^{-7}$ & 4 \\ \hline \hline
	\end{tabular}
	\caption{\textit{Convergence rate of the spatial numerical discretization defined in Eq.~\eqref{eq_4th_space_variable}. The relative error is calculated in $L^1,L^2,L^\infty$--norms, for a fixed $\Delta t_{\rm ref} = 10^{-5}$ and $t_{\rm fin} = 0.1$. The domain is $\Omega = [-1,1]^2$, the initial condition is defined in Eq.~\eqref{eq_expr_IC_tests}  with $x_{m_1} = y_{m_1} = 0$ and $\sigma = 0.1$, the velocity in Eq.~\eqref{eq_expr_velocity_space}, and homogeneous Dirichlet boundary conditions (i.e., $f = 0$ in Eq.~\eqref{eq_Dir_bc}).}}
	\label{table_4th_space_variable}
\end{table}

In Fig.~\ref{fig_4th_space_variable} (a), we show the fourth order space accuracy of the numerical scheme defined in Eq.~\eqref{eq_4th_space}. To calculate the error, we add an artificial term {(see for example \cite{roache2002code})} at the numerical scheme in \eqref{eq_4th_space}, as follows 
\begin{align}
	\label{eq_4th_space_F}
	\displaystyle  \frac{c_h^{n+1} - c_h^n}{\Delta t} = & \frac{1}{2} \left(L^{\rm 4th}_h + D_h^{\rm 4th}\right) (c_h^n + c_h^{n+1})  \\ \nonumber & +
 \frac{1}{2} F(c^{\rm exa}(x_h,y_h,t^n))+F(c^{\rm exa}(x_h,y_h,t^{n+1}))
\end{align}
where the time derivative is discretized with a second order Crank-Nicolson method, with $c_h^{n+1}\approx c_h(t^{n+1})$, the expression of $F$ is the following 
\begin{equation}
	F(c^{\rm exa}(x,y,t)) = \frac{\partial}{\partial t} c^{\rm exa}(x,y,t) - \Delta\, c^{\rm exa}(x,y,t) - u\nabla \cdot c^{\rm exa}(x,y,t).
\end{equation}
and the one for the manufactured exact solution is 
\begin{align} \label{eq_manif_exact}
	\displaystyle
	c_h^{\rm exa}(t) = &\cos(t)\exp\left(-\frac{(x-x_{m_1})^2+(y-y_{m_1})^2}{2\sigma^2}\right) + \sin(t)\exp\left(-\frac{(x-x_{m_2})^2+(y-y_{m_2})^2}{2\sigma^2}\right).
\end{align}
After defining the expression of $F$, we calculate the following error
\begin{equation}
	{\rm error}_\beta = \frac{||c_h(t_{\rm fin}) - c^{\rm exa}(x_h,y_h,t)||_\beta}{||c^{\rm exa}(x_h,y_h,t)||_\beta}
\end{equation}
where $\beta = 1,2,\infty$ indicates which norm we are considering. In Fig.~\ref{fig_4th_space_variable} (a) and in Table~\ref{table_4th_space}, we show the numerical error of the space discretization in Eq.~\eqref{eq_4th_space} and the slope is $\approx 3.95$.

Analogously, in Fig.~\ref{fig_4th_space_variable} (b), we show the order of accuracy of the space discretization in Eq.~\eqref{eq_4th_space_variable}. In this case, to calculate the solution at final time we consider the following
\begin{align}
	\label{eq_4th_space_variable_F}
	\displaystyle  \frac{c_h^{n+1} - c_h^n}{\Delta t} = & \frac{1}{2} \left(L^{\rm 4th}_h + Q_h^{\rm 4th}\right) (c_h^n + c_h^{n+1}) \\ \nonumber &+ \frac{1}{2}F(c^{\rm exa}(x_h,y_h,t^n))+F(c^{\rm exa}(x_h,y_h,t^{n+1})).
\end{align}
In this case, the expression for $F$ is 
\begin{equation*}
	F(c_h^{\rm exa}(x,y,t)) = \frac{\partial}{\partial t} c_h^{\rm exa}(x,y,t) - \Delta\, c_h^{\rm exa}(x,y,t) - \nabla \cdot (\vec{u}_h\, c_h^{\rm exa}(x,y,t)),
\end{equation*}
and the expression for $c^{\rm exa}$ is in Eq.~\eqref{eq_manif_exact}. Here we define the expression of the two components of the velocity vector
\begin{equation}
	\label{eq_expr_velocity_space}
	u^x = x^3, \quad u^y = y^3,
\end{equation} 
and the initial condition
\begin{equation}
	\label{eq_expr_IC_tests}
	c_h^0 = c_h(t = 0) = \exp\left(-\frac{(x-x_{m_1})^2+(y-y_{m_1})^2}{2\sigma^2}\right) \equiv c_h^{\rm exa}(t = 0), 
\end{equation}
and in Fig.~\ref{fig_4th_space_variable} (b) and Table~\ref{table_4th_space_variable}, we show the numerical error of the space discretization in Eq.~\eqref{eq_4th_space_variable}, and the slope is $\approx 3.95$.

To calculate the error of the time discretization in {Eqs.~\eqref{eq_scheme_3rd}}, we first compute a reference solution $c_{\varepsilon,h}^{\rm ref}$, choosing a reference time step $\Delta t_{\rm ref}$, number of points $N_{\rm ref}$ and final time $t_{\rm fin}$, for the following set of $\varepsilon = 10^{-k}, k \in \{1,2,3,4,5,6\}$. Then, we calculate different solutions $c^{\Delta t,N}_{\varepsilon,h}$, for different $\Delta t = 0.01\cdot 2^{-k}, k\in \{0,1,2,3,4,5,6\}$ and $N = 20,40,80,160$. After computing all the solutions, we calculate the $L^2-$norm of the relative error as follow:
\begin{equation}
\label{eq_error}
	{\rm error} = \frac{||c_{\varepsilon,h}^{\rm ref} - c^{\Delta t,N}_{\varepsilon,h}||_2}{||c_{\varepsilon,h}^{\rm ref}||_2}.
\end{equation}

\begin{figure}[!ht]
	\centering
	\centering
	\hfill
	\begin{minipage}[b]
		{.49\textwidth}
		\centering
		\begin{overpic}[abs,width=\textwidth,unit=1mm,scale=.25]{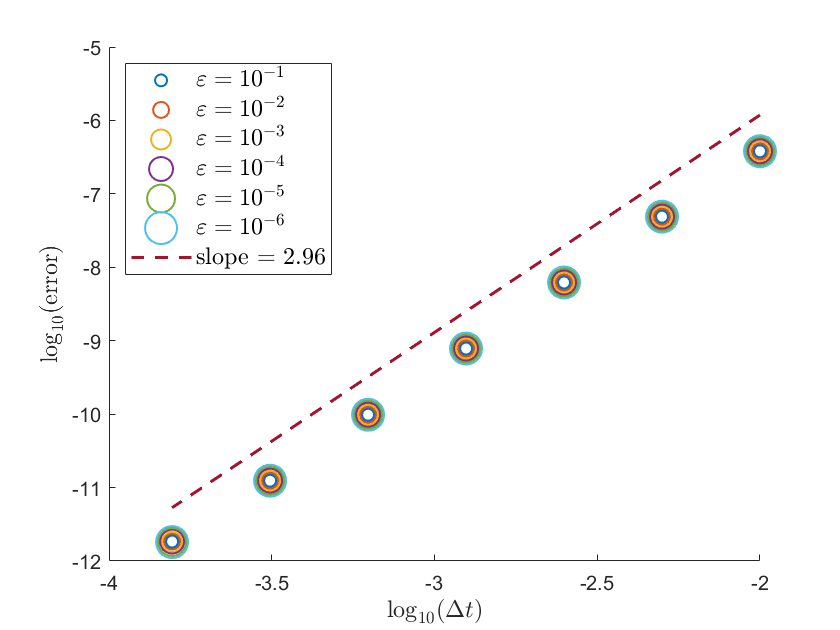}
			\put(0.,44){(a)}
		\end{overpic}	
	\end{minipage}\hfill
	\begin{minipage}[b]
		{.49\textwidth}
		\centering
		\begin{overpic}[abs,width=\textwidth,unit=1mm,scale=.25]{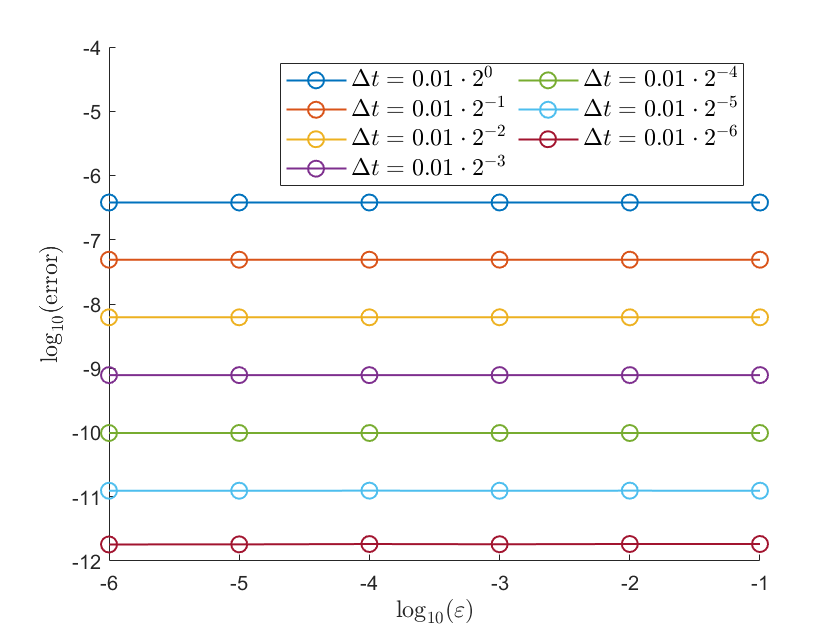}
			\put(0.,44){(b)}
		\end{overpic}	
	\end{minipage}
	\caption{\textit{Plot of the relative $L^2$ error for the third order scheme in {Eqs.~\eqref{eq_scheme_3rd}} as a function of $\Delta t$ and different $\varepsilon$ (a) and as a function of $\varepsilon$ and different $\Delta t$ (b). The parameters of the tests are: $t_{\rm fin}=0.1, \Delta t_{\rm ref} = 10^{-5},  N_{\rm ref} = 160, D = 0.02, \delta = 10^{-2}$. The domain is $\Omega = [-1,1]^2$, the initial condition is defined in Eq.~\eqref{eq_expr_IC_tests}  with $x_{m_1} = y_{m_1} = 0$ and $\sigma = 0.1$, and the velocity expression in Eq.~\eqref{eq_expr_velocity_2}, $A = 10$.}}
	\label{fig_time_3rd_order_2}
\end{figure}

In Figs.~\ref{fig_time_3rd_order}--\ref{fig_time_3rd_order_2} we show the $L^2-$norm of the error, as a function of $\Delta t$ and for different values of $\varepsilon$ in panel (a), and {in} function of $\varepsilon$ and different $\Delta t$ in panel (b). The method that we consider here is the third order accurate numerical scheme defined in {Eqs.~\eqref{eq_scheme_3rd}}, to show that the numerical scheme is uniformly accurate in $\varepsilon$. The considered expressions for the velocity are the following
\begin{eqnarray}
	\label{eq_expr_velocity}
	u^x = A\,x^3\cos(2\pi\,t/\varepsilon), \quad u^y = A\,y^3\cos(2\pi\,t/\varepsilon),\\
	u^x  = A\,\cos(2\pi\,t/\varepsilon)\frac{x}{{x^2 + y^2 + \gamma}}, \quad u^y = A\,\cos(2\pi\,t/\varepsilon)\frac{y}{{x^2 + y^2 + \gamma}},
	\label{eq_expr_velocity_2}
\end{eqnarray}
where $\varepsilon\in\{10^{-1},10^{-2},10^{-3},10^{-4},10^{-5},10^{-6} \}$ and $\gamma = 0.1$ when we consider a regular squared domain, while $\gamma = 0$ when we consider a perforated domain.
In Table~\ref{table_time_3rd_order}, we show the  rate of convergence of the numerical scheme in {Eqs.~\eqref{eq_scheme_3rd}}, at final step $t = 0.1$, for three values of $\varepsilon = 10^{-2},10^{-4},10^{-6}$, and $N_{\rm ts}$ is the number of time steps that we choose.

\begin{figure}[H]
	\centering
	\hfill
	\begin{minipage}
		{.45\textwidth}
		\centering		%\includegraphics[width=\textwidth]{Figures/accuracy_time_indip_eps.png}
		\begin{overpic}[abs,width=\textwidth,unit=1mm,scale=.25]{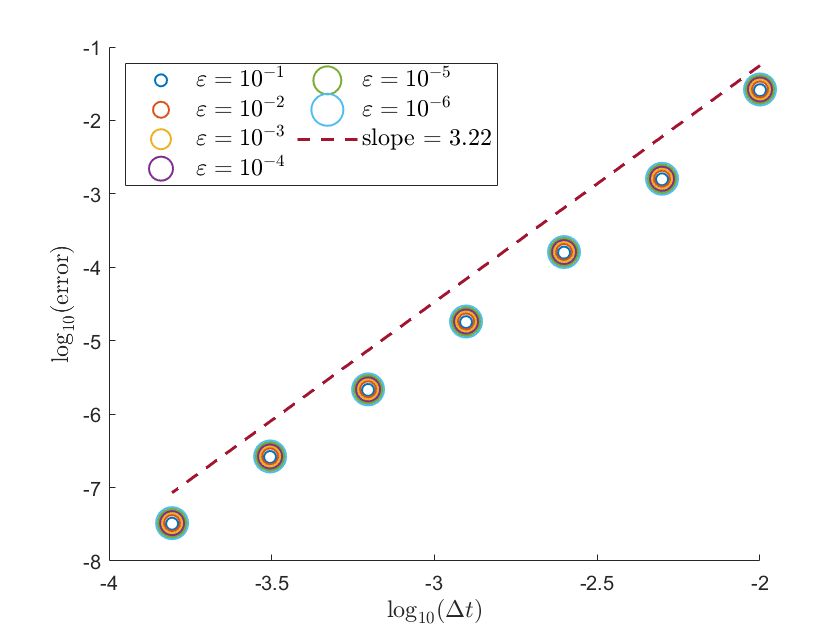}
			\put(0.,44){(a)}
		\end{overpic}	
	\end{minipage}
	\begin{minipage}
		{.45\textwidth}
		\centering
		\begin{overpic}[abs,width=\textwidth,unit=1mm,scale=.25]{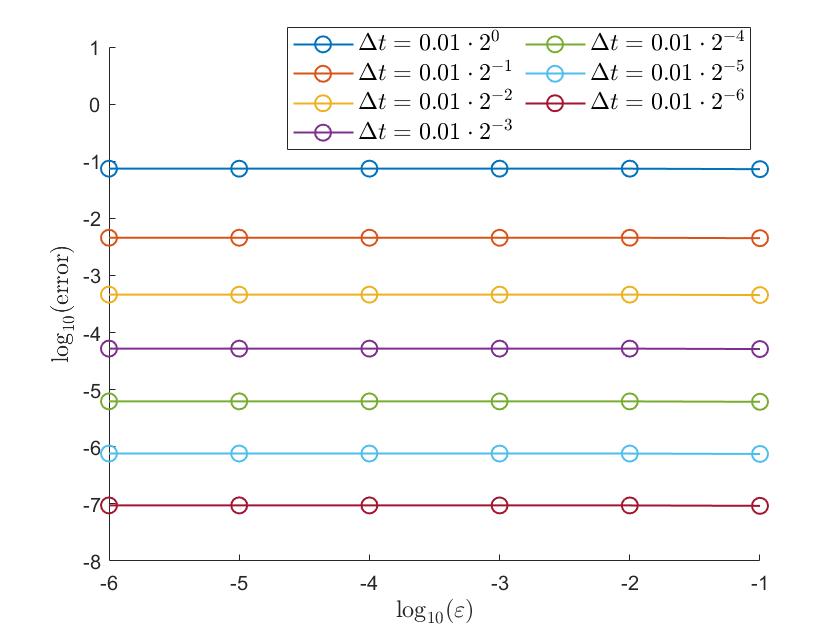}
			\put(0.,44){(b)}
		\end{overpic}
	\end{minipage}
	\hspace*{\fill}
	\caption{\textit{Plot of the relative error in $L^2$ norm for the third order scheme in {Eqs.~\eqref{eq_scheme_3rd}} as a function of $\Delta t$ and different $\varepsilon$ (a) and as a function of $\varepsilon$ and different $\Delta t$ (b). The parameters of the tests are: $t_{\rm fin}=0.1, \Delta t_{\rm ref} = 10^{-5},  N_{\rm ref} = 160, D = 0.01, \delta = 10^{-2}$. The domain is $\Omega = [-1,1]^2\setminus\mathcal{B}$, the initial condition is defined in Eq.~\eqref{eq_expr_IC_tests}  with $x_{m_1} = y_{m_1} = 0.35$ and $\sigma = 0.1$, and the velocity expression in Eq.~\eqref{eq_expr_velocity_2}, $A = 1$.}}
	\label{fig_accuracy_3rd_order_ghost}
\end{figure}
In Fig.~\ref{fig_accuracy_3rd_order_ghost}, we show the $L^2-$norm of the relative error, as a function of $\Delta t$ and for different values of $\varepsilon$ in panel (a), and in function of $\varepsilon$ and different $\Delta t$ in panel (b). Here, we are examining a numerical scheme that achieves third order accuracy in time, defined in {Eqs.~\eqref{eq_scheme_3rd}}, to show that it is uniformly accurate in $\varepsilon$. The domain considered is $\Omega = [0,1]^2\setminus \mathcal{B}$, and the space discretization is defined in Eq.~\eqref{eq_ghost_linear_system}. The expression for the velocity is in Eq.~\eqref{eq_expr_velocity_2}. 
Analogously, in Fig.~\ref{fig_accuracy_3rd_order_ghost_dx}, we show the $L^2-$norm of the relative error, as a function of $\Delta x (=h)$ and for different values of $\varepsilon$ in panel (a), and in function of $\varepsilon$ and different $\Delta x$ in panel (b). Considering the ghost points, it is more challenging to prove the uniform accuracy in space, for two reasons: an additional parameter appears in the system, the thickness $\delta$ of the bubble $\mathcal B$, and an interpolation stencil of 16 points is considered when we solve the linear system for the ghost points. For these reasons, we consider smaller final time $t_{\rm fin} = 10^{-4}$ and $\Delta t_{\rm ref} = 10^{-6}$, and  %a smaller $\Delta t_{\rm ref}$ to prove the uniform accuracy in space. In Fig.~\ref{fig_accuracy_3rd_order_ghost_dx} we choose a final time $t_{\rm fin} = 10^{-4}$ and a $\Delta t_{\rm ref} = 10^{-6}$. The space discretization considered is described in Eqs.~\eqref{eq_ghost_linear_system}, and in this case, 
we do not include the tests for $N = 20$ and $N = 40$, (as we did in the other tests) since there are not enough points in the domain $\Omega_h$ to create the structure of the 16-point stencil for each ghost point. 

The results in Fig.~\ref{fig_accuracy_3rd_order_ghost} show a significant improvement respect to those obtained in \cite{astuto2023time}. In that study, the time numerical scheme is uniformly accurate in $\varepsilon$, and there is no interaction between $\varepsilon$ and the time step $\Delta t$, defined in the method. By the way, the space discretization affects the values of the errors. The results are shown  for a fixed space step $\Delta x$, highlighting a reduction of the errors for $\varepsilon < \Delta x$. We affirm that the proposed numerical scheme is uniform accurate in time, with a strong influence of the space discretization that shows that for $\varepsilon > \Delta x$, the values of the errors are larger then the other cases. The results presented here provide evidences of the lack of correlation between $\varepsilon $ and $\Delta t$, indeed, with a fourth order space discretization, the uniform accuracy in time is much better performed, for all $\varepsilon$.

\begin{figure}[H]
	\centering
	\hfill
	\begin{minipage}
		{.45\textwidth}
		\centering		%\includegraphics[width=\textwidth]{Figures/accuracy_time_indip_eps.png}
		\begin{overpic}[abs,width=\textwidth,unit=1mm,scale=.25]{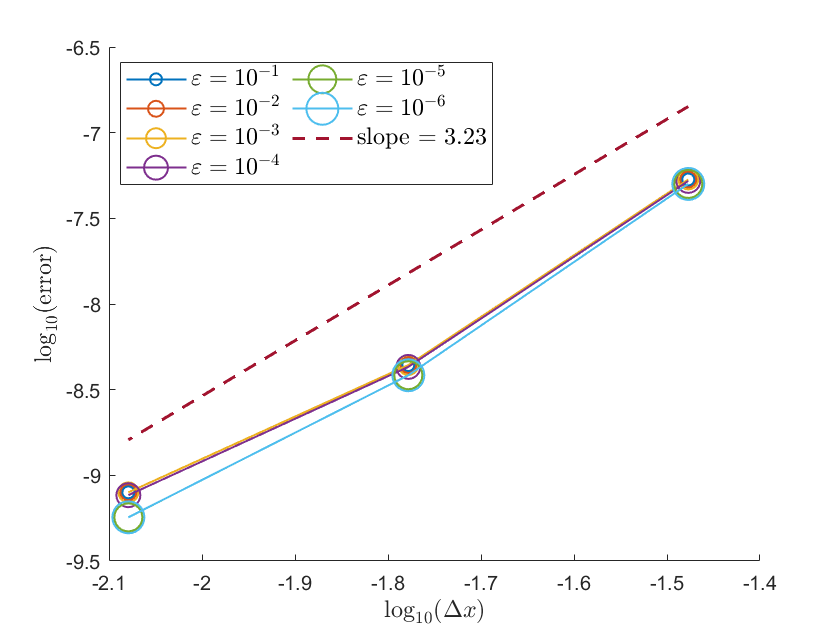}
			\put(0.,48){(a)}
		\end{overpic}	
	\end{minipage}
	\begin{minipage}
		{.45\textwidth}
		\centering
		\begin{overpic}[abs,width=\textwidth,unit=1mm,scale=.25]{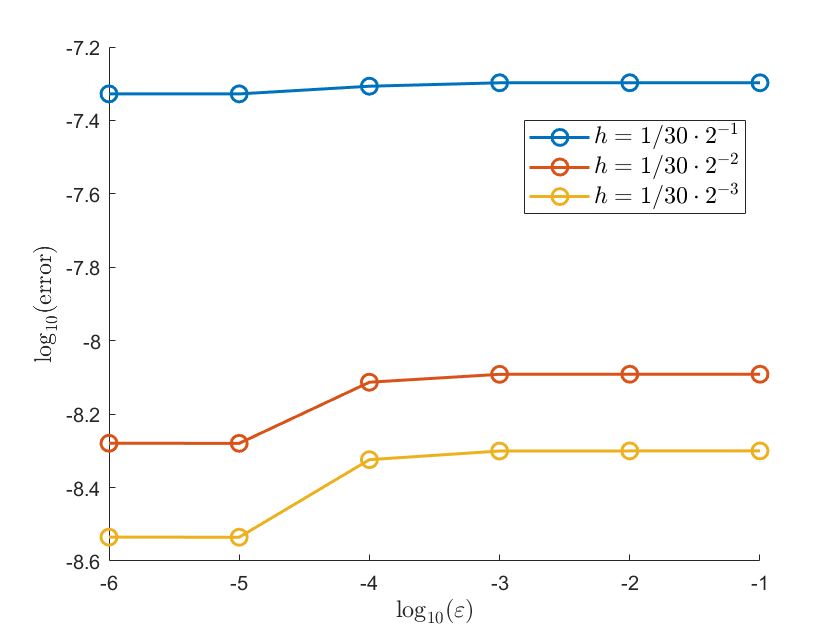}
			\put(0.,48){(b)}
		\end{overpic}
	\end{minipage}
	\hspace*{\fill}
	\caption{\textit{Plot of the relative error in $L^2$ norm for the third order scheme in {Eqs.~\eqref{eq_scheme_3rd}} as a function of ${\Delta x}$ and different $\varepsilon$ (a) and as a function of $\varepsilon$ and different ${\Delta x}$ (b). The parameters of the tests are: $t_{\rm fin}=10^{-4}, \Delta t_{\rm ref} = 10^{-6},  N_{\rm ref} = 480, D = 0.01, \delta = 10^{-2}$. The domain is $\Omega = [-1,1]^2\setminus\mathcal{B}$, the initial condition is defined in Eq.~\eqref{eq_expr_IC_tests}  with $x_{m_1} = y_{m_1} = 0.3$ and $\sigma = 0.1$, and the velocity expression in Eq.~\eqref{eq_expr_velocity_2}, $A = 1$.}}
\label{fig_accuracy_3rd_order_ghost_dx}
\end{figure}
\begin{figure}[H]
	\centering
	\hfill
	\begin{minipage}
		{.45\textwidth}	\centering
		\begin{overpic}[abs,width=\textwidth,unit=1mm,scale=.25]{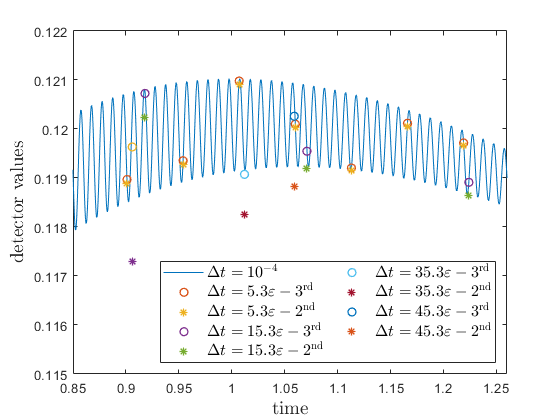}
			\put(0.,48){(a)}
			\put(22.,49){$\varepsilon = 10^{-2}$}
		\end{overpic}
	\end{minipage}\hfill
	\begin{minipage}
		{.45\textwidth}	\centering
		\begin{overpic}[abs,width=\textwidth,unit=1mm,scale=.25]{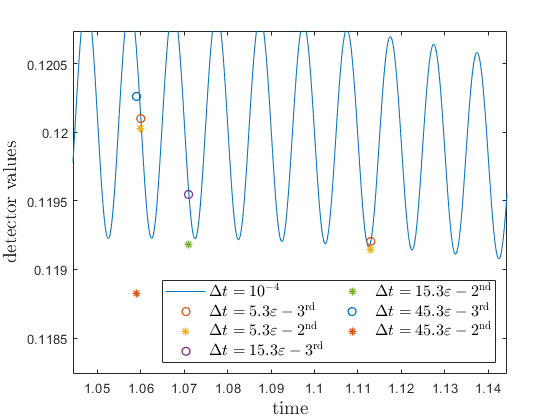}
			\put(0.,48){(b)}
			\put(22.,49){$\varepsilon = 10^{-2}$}
		\end{overpic}
	\end{minipage}
	\hspace*{\fill}
	\caption{\textit{Comparison of the values of the concentration $c_{\varepsilon,h}$ at the detector, between the $2^{\rm nd}$ order scheme (see {Eqs.\eqref{eq_scheme_2nd}}) and the $3^{\rm rd}$ order scheme (see {Eqs.~\eqref{eq_scheme_3rd}}). We show different values of $\Delta t$ (a), and a zoom-in in panel (b). The domain is $\Omega = [-1,1]^2\setminus\mathcal{B}$, the velocity defined in Eq.~\eqref{eq_expr_velocity_2}, the initial condition is defined in Eq.~\eqref{eq_expr_IC_tests} and the detector is centered in $P=(0,-0.5)$, $P\in\Omega$. The parameters of the test are: $\varepsilon = 10^{-2}, N = 80, A = 100, \delta = 10^{-3}, x_{m_1} = 0, y_{m_1} = 0$ and $D = 0.01$.}}
	\label{fig_detector_cos_1em2}
\end{figure}
Figs.~\ref{fig_detector_cos_1em2}--\ref{fig_detector_cos_1em3} represent a qualitative comparison between the second (asterisks in the plots, see {Eqs.\eqref{eq_scheme_2nd}}{)} and the third order numerical scheme (circles in the plots, see {Eqs.~\eqref{eq_scheme_3rd}} of the time evolution of the solution at the detector located in $P\in \Omega$. The expression for the velocity is defined in Eq.~\eqref{eq_expr_velocity_2}, with $A = 100$. We choose two different values of $\varepsilon$: in Fig.~\ref{fig_detector_cos_1em2}, $\varepsilon = 0.01$ and in Fig.~\ref{fig_detector_cos_1em3} $\varepsilon = 0.001$. We show a reference solution, $c_{\varepsilon,h}^{\rm ref}$ (blue line) with $\Delta t_{\rm ref} = 10^{-4}$, together with different solutions for different time steps.
\begin{figure}[H]
	\centering
	\hfill
	\begin{minipage}
		{.45\textwidth}	\centering
		\begin{overpic}[abs,width=\textwidth,unit=1mm,scale=.25]{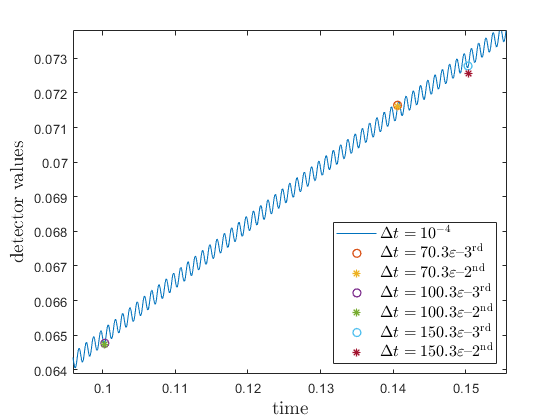}
			\put(0.,48){(a)}
			\put(22.,49){$\varepsilon = 10^{-3}$}
		\end{overpic}
	\end{minipage}\hfill
	\begin{minipage}
		{.45\textwidth}	\centering
		\begin{overpic}[abs,width=\textwidth,unit=1mm,scale=.25]{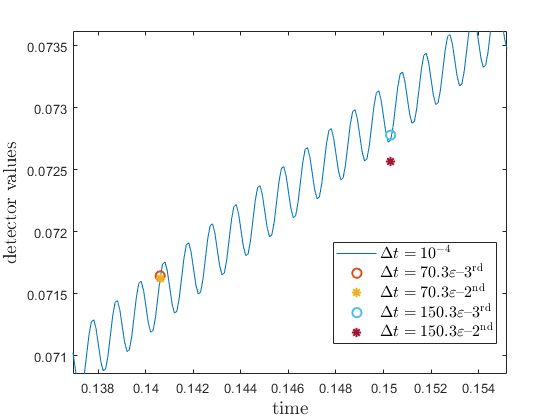}
			\put(0.,48){(b)}
			\put(22.,49){$\varepsilon = 10^{-3}$}
		\end{overpic}
	\end{minipage}
	\hspace*{\fill}
	\caption{\textit{Comparison of the values of the concentration $c_{\varepsilon,h}$ at the detector, between the $2^{\rm nd}$ order scheme (see {Eqs.\eqref{eq_scheme_2nd}}) and the $3^{\rm rd}$ order scheme (see {Eqs.~\eqref{eq_scheme_3rd}}). We show different values of $\Delta t$ (a), and a zoom-in in panel (b). The domain is $\Omega = [-1,1]^2\setminus\mathcal{B}$, the velocity defined in Eq.~\eqref{eq_expr_velocity_2}, the initial condition is defined in Eq.~\eqref{eq_expr_IC_tests} and the detector is centered in $P=(0,-0.5)$, $P\in\Omega$. The parameters of the test are: $\varepsilon = 10^{-3}, N = 80, A = 100, \delta = 10^{-3}, x_{m_1} = 0, y_{m_1} = 0$ and $D = 0.01$.}}
	\label{fig_detector_cos_1em3}
\end{figure}

 In Fig.~\ref{fig_cpu} we show the advantage in CPU time when using the third order accurate numerical scheme. For a given computational time, the third order scheme shows that its accuracy is at least one order of magnitude higher then that of the second order scheme. We show the error defined in Eq.~\eqref{eq_error},  with $N = 20$, $\Delta t_{\rm ref} = 10^{-4}$ and \\$\Delta t \in \{1/5, 1/10, 1/20, 1/30, 1/40, 1/50, 1/60, 1/70, 1/80, 1/90, 1/100\}$. Finally, in Fig.~\ref{fig_detector_cos_1em3_ghost} we show the evolution in time of the solution of the numerical scheme defined in {Eqs.~\eqref{eq_scheme_3rd}}, whose space discretization is defined in Eqs.~\eqref{eq_ghost_linear_system}, in presence of ghost points. As before, we evaluate the solution in a point $P\in \Omega$, and the velocity is defined in Eq.~\eqref{eq_expr_velocity_2}, with $A = 100$.

\subsection*{Arbitrary domains and level-set functions}
In this subsection, we focus on different geometries. We aim to solve an advection-diffusion equation in arbitrary domains with homogeneous Dirichlet boundary conditions. The system reads
\begin{eqnarray} 
\label{eq:Dirichlet_arbitrary}
	\left\{
	\begin{array}{l}
		\displaystyle     \frac{\partial c_\varepsilon}{\partial t} = D\Delta c_\varepsilon + \nabla \cdot (c_\varepsilon \vec{u}(t/\varepsilon))  \quad \text{ in $\Omega$} \\
		\displaystyle c_\varepsilon = 0 \quad  \text{ on $\Gamma$}
	\end{array}
	\right.
\end{eqnarray}
where the velocity $\vec{u}(t/\varepsilon) = \vec{u}(x,y,t/\varepsilon)$ is given by Eq.~\eqref{eq_expr_velocity_2}, with $\gamma = 0.1$. Here,  Dirichlet boundary conditions are employed, eliminating the need for interpolating the spatial derivatives of the numerical solution as we did in Eq.~\eqref{QHghost}. Adopting the same 16-point stencil in Fig.~\ref{stencil} (b), this time we prove fourth order accuracy for the space discretization.

The expression of the different level-set functions that consider are the following:

\subsubsection*{Ellipsoidal domain}
Let us consider the following level-set function
\[\phi(x,y) = \frac{(X - \sqrt{2}/20)^2}{0.7^2} + \frac{(Y - \sqrt{3}/30)^2}{0.45^2} - 1, \]
where $X = \cos(\pi/6)x-\sin(\pi/6)y$, and $Y = \cos(\pi/6)x + \sin(\pi/6)y.$
%\[ \]

\subsubsection*{Flower-shape domain}
Let us consider the following level-set function
\begin{align*}
\phi(x,y) = \frac{R - 0.52 - (Y^5 +5X^4Y-10X^2Y^3)}{5R^5}
\end{align*}
where $X = x-0.03\sqrt 3,\, Y = y-0.04\sqrt 2,$ and $R = \sqrt{X^2+Y^2}.$

\begin{figure}[H]
    \centering
    \includegraphics[width=0.65\textwidth]{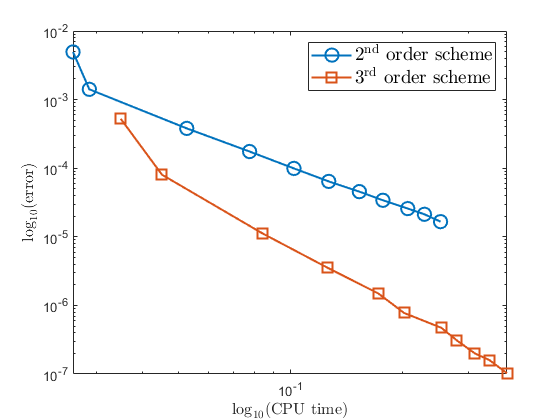}
    \caption{\textit\textit{Comparison of the CPU time between the $2^{\rm nd}$ order scheme (blue circles, see {Eqs.\eqref{eq_scheme_2nd}}) and the $3^{\rm rd}$ order scheme (red squares, see {Eqs.~\eqref{eq_scheme_3rd}}). We show the error defined in Eq.~\eqref{eq_error} with $\Delta t_{\rm ref} = 10^{-5}$ for different $\Delta t \in \{1/5, 1/10, 1/20, 1/30, 1/40, 1/50, 1/60, 1/70, 1/80, 1/90, 1/100\}$.
    The parameters of the test are: $\varepsilon = 10^{-3}, N = 20, A = 100, \delta = 10^{-3}, t_{\rm fin} = 1$ and $D = 0.01$.}}
    \label{fig_cpu}
\end{figure}
\subsubsection*{Cardioid-shape domain}
Let us consider the following level-set function
\[\phi(x,y) = \left(3\left( X^2 + Y^2 \right) -X \right)^2 - X^2 - Y^2, \]
where $X = x-0.04\sqrt{3}-0.35$ and $Y = y-0.05\sqrt{2}.$

In Fig.\ref{fig_arbitrary_domain_results}
we show the numerical results when solving Eq.~\eqref{eq:Dirichlet_arbitrary}, for the arbitrary domains: ellipsoidal (first column), flower-shaped (second column) and cardioid-shaped domains (third column). We have the $L^2-$norm of the error, as a function of $\Delta x$ and for different values of $\varepsilon$ in panel (a)-(c), and in function of $\varepsilon$ and different $\Delta x$ in panel (d)-(f).  We remark here that the order of accuracy in space is four due to the choice of Dirichlet boundary conditions. We remind that with the boundary conditions defined in Eq.~\eqref{QHghost}, the space discretization was third order. In panels (g)-(i), we show the detector values of the time evolution of the solution at the detector located in $P\in \Omega$, with $\varepsilon = 0.001$.  We show a reference solution, $c_{\varepsilon,h}^{\rm ref}$ (blue line) with $\Delta t_{\rm ref} = 10^{-4}$, together with different solutions for different time steps. {The method is implemented in Matlab on a Dell Inspiron 13-5379, 8th Generation Intel Core i7, 16GB RAM.
}

% \todo[inline]{aggiungere test con ghost}

%%%%%%%%%%%%%%%%%%%%%%%%%% HERE

\begin{figure}[htp]
	\centering
	\hfill
	\begin{minipage}
		{.45\textwidth}	\centering
		\begin{overpic}[abs,width=\textwidth,unit=1mm,scale=.25]{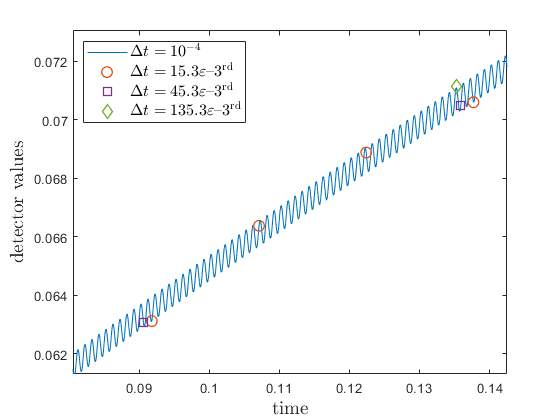}
			\put(0.,48){(a)}
			\put(22.,49){$\varepsilon = 10^{-3}$}
		\end{overpic}
	\end{minipage}\hfill
	\begin{minipage}
		{.45\textwidth}	\centering
		\begin{overpic}[abs,width=\textwidth,unit=1mm,scale=.25]{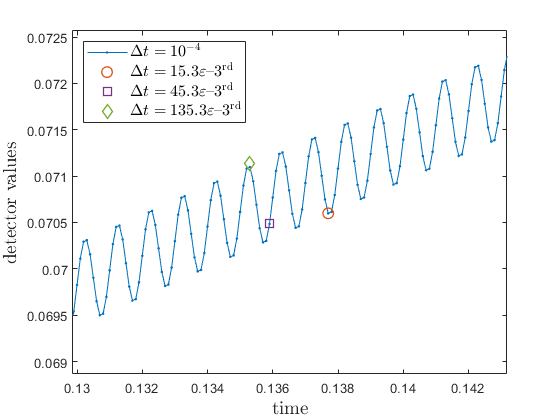}
			\put(0.,48){(b)}
			\put(22.,49){$\varepsilon = 10^{-3}$}
		\end{overpic}
	\end{minipage}
	\hspace*{\fill}
	\caption{\textit{Detector values of the concentration $c_{\varepsilon,h}$ for the $3^{\rm rd}$ order numerical scheme in {Eqs.~\eqref{eq_scheme_3rd}}. We show different values of $\Delta t$ (a), and a zoom-in in panel (b). The domain is $\Omega = [-1,1]^2\setminus \mathcal B$, the velocity defined in Eq.~\eqref{eq_expr_velocity_2}, the initial condition is defined in Eq.~\eqref{eq_expr_IC_tests} and the detector is centered in $P=(0.35,0.35)$, $P\in\Omega$. The parameters of the test are: $\varepsilon = 10^{-3}, N = 80, A = 100, \delta = 10^{-3}, x_{m_1} =  y_{m_1} = 0.5$ and $D = 0.01$.}}
	\label{fig_detector_cos_1em3_ghost}
\end{figure}

\begin{figure}[H]
	\centering
	\begin{minipage}[b]
		{.33\textwidth}
		\centering
		\begin{overpic}[abs,width=1.\textwidth,unit=1mm,scale=.25]{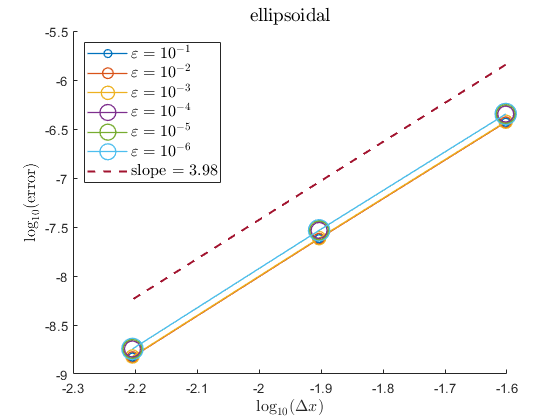}
			\put(0,36){(a)}
		\end{overpic}
	\end{minipage}\hfill
	\begin{minipage}[b]{.33\textwidth}
		\begin{overpic}[abs,width=\textwidth,unit=1mm,scale=.25]{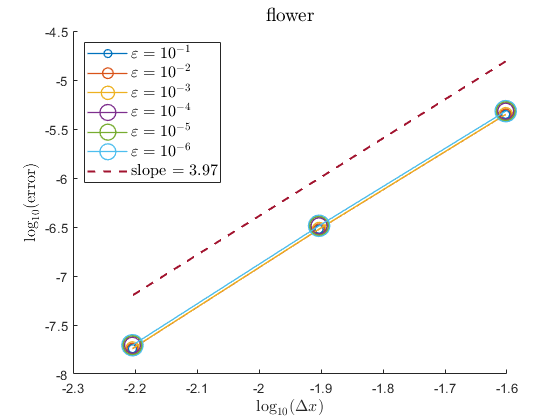}
			\put(0,36){(b)}
		\end{overpic}
	\end{minipage}
 \begin{minipage}[b]{.33\textwidth}
		\begin{overpic}[abs,width=\textwidth,unit=1mm,scale=.25]{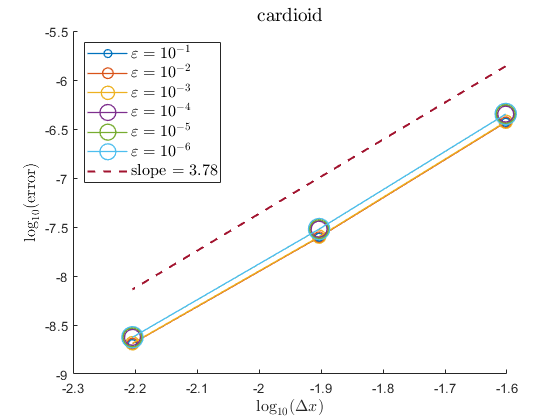}
			\put(0,36){(c)}
		\end{overpic}
	\end{minipage}
 \begin{minipage}[b]
		{.33\textwidth}
		\centering
		\begin{overpic}[abs,width=1.\textwidth,unit=1mm,scale=.25]{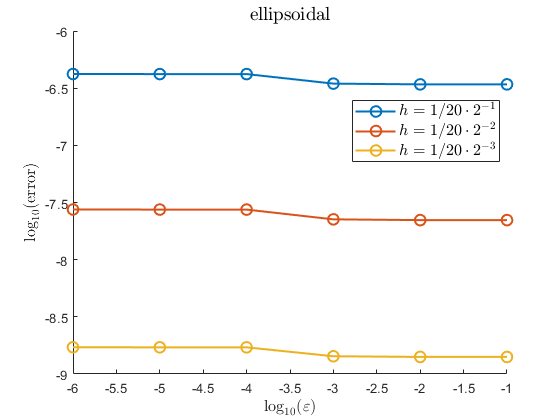}
			\put(0,36){(d)}
		\end{overpic}
	\end{minipage}\hfill
	\begin{minipage}[b]{.33\textwidth}
		\begin{overpic}[abs,width=\textwidth,unit=1mm,scale=.25]{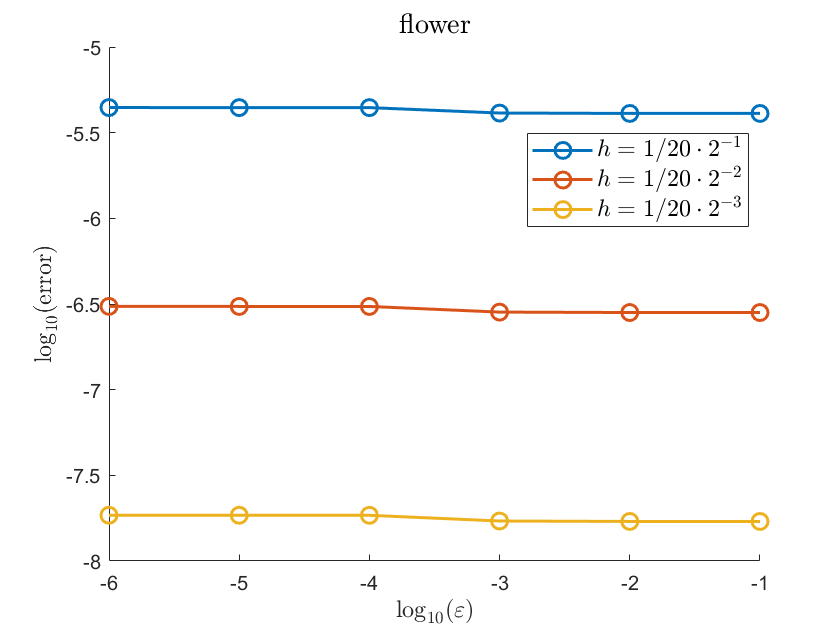}
			\put(0,36){(e)}
		\end{overpic}
	\end{minipage}
 \begin{minipage}[b]{.33\textwidth}
		\begin{overpic}[abs,width=\textwidth,unit=1mm,scale=.25]{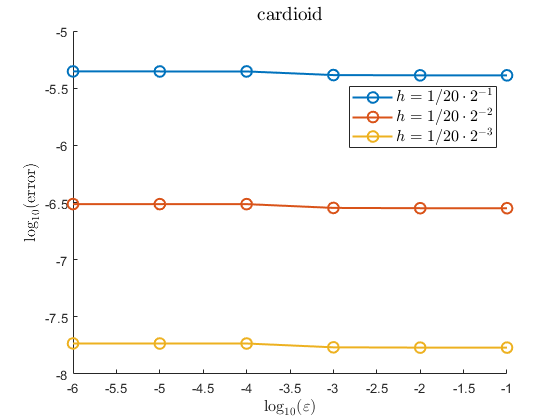}
			\put(0,36){(f)}
		\end{overpic}
	\end{minipage}
  \begin{minipage}[b]
		{.33\textwidth}
		\centering
		\begin{overpic}[abs,width=1.\textwidth,unit=1mm,scale=.25]{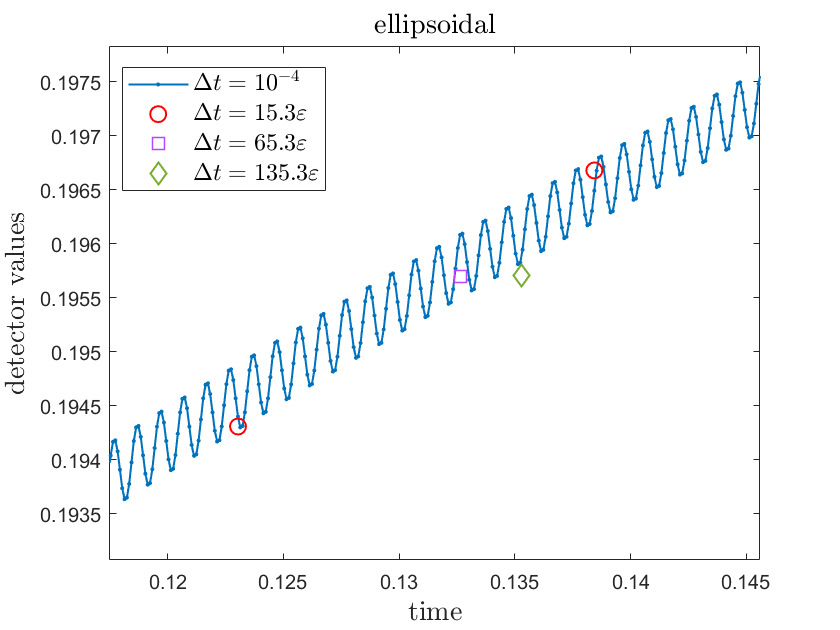}
			\put(0,36){(g)}
		\end{overpic}
	\end{minipage}\hfill
	\begin{minipage}[b]{.33\textwidth}
		\begin{overpic}[abs,width=\textwidth,unit=1mm,scale=.25]{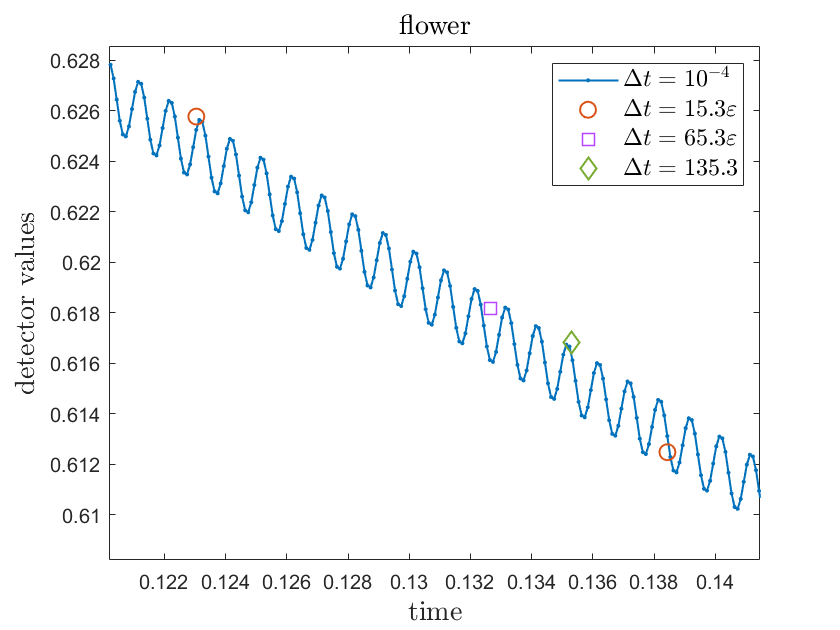}
			\put(0,36){(h)}
		\end{overpic}
	\end{minipage}
 \begin{minipage}[b]{.33\textwidth}
		\begin{overpic}[abs,width=\textwidth,unit=1mm,scale=.25]{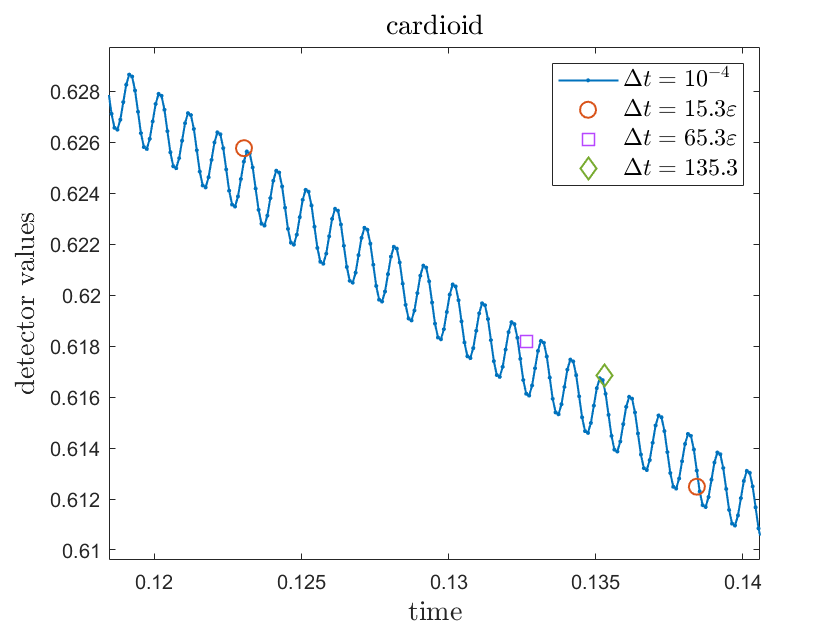}
			\put(0,36){(i)}
		\end{overpic}
	\end{minipage}
	\caption{\textit{As shown in the title of each panel, this figure shows different features for three different domains: ellipsoidal, flower-shaped and cardioid-shaped domains (see Eq.~\eqref{eq:Dirichlet_arbitrary}). We plot the relative error in space in $L^2$ norm, as a function of $h$ and different $\varepsilon$ in (a)-(c) and as a function of $\varepsilon$ and different $h$ in (d)-(f). The parameters of these tests are: $t_{\rm fin}=10^{-4}, \Delta t_{\rm ref} = 10^{-6},  N_{\rm ref} = 640, D = 0.01, \delta = 10^{-2}$. The initial condition is defined in Eq.~\eqref{eq_expr_IC_tests}  with $x_{m_1} = y_{m_1} = 0$ and $\sigma = 0.1$, and the velocity expression in Eq.~\eqref{eq_expr_velocity_2}, with $A = 1$ in (a)-(f). In panels (g)-(i), we show the detector values of the concentration $c_{\varepsilon,h}$, and the detector is centered in $P=(0,0)$, $P\in\Omega$. The parameters of the test are: $\varepsilon = 10^{-3}, N = 80, A = 100, \delta = 10^{-3}, x_{m_1} =  y_{m_1} = 0$ and $D = 0.01$.}}
	\label{fig_arbitrary_domain_results}
\end{figure}

\begin{landscape}
\begin{table}[h]
	\centering
	\begin{tabular}{|c|c|c|c|c|c|c|c|}
		\hline \hline
		$\varepsilon = 10^{-2}$   \\ \hline
		${\rm N}_{\rm ts}$ & $10\cdot 2^0$  & $10\cdot 2^1$  & $10\cdot 2^2$  & $10\cdot 2^3$ & $10\cdot 2^4$  & $10\cdot 2^5$ & $10\cdot 2^6$ \\ \hline
		{\rm error} & 3.805$\cdot 10^{-7}$ & 4.894$\cdot 10^{-8}$ & 6.209$\cdot 10^{-9}$ & 7.820$\cdot 10^{-10}$ & 9.820$\cdot 10^{-11}$ & 1.241$\cdot 10^{-11}$ & 1.819$\cdot 10^{-12}$ \\ \hline       {\rm order} & -- & 2.96 & 2.98 & 2.99 & 2.99 & 2.98 & 2.77 \\
		\hline \hline    
		$\varepsilon = 10^{-4}$   \\ \hline
		${\rm N}_{\rm ts}$ & $10\cdot 2^0$  & $10\cdot 2^1$  & $10\cdot 2^2$  & $10\cdot 2^3$ & $10\cdot 2^4$  & $10\cdot 2^5$ & $10\cdot 2^6$ \\ \hline
		{\rm error} & 3.801$\cdot 10^{-7}$ & 4.894$\cdot 10^{-8}$ & 6.209$\cdot 10^{-9}$ & 7.820$\cdot 10^{-10}$ & 9.820$\cdot 10^{-11}$ & 1.241$\cdot 10^{-11}$ & 1.816$\cdot 10^{-12}$ \\ \hline       {\rm order} & -- & 2.96 & 2.98 & 2.99 & 2.99 & 2.98 & 2.77 \\
		\hline \hline    
		$\varepsilon = 10^{-6}$   \\ \hline
		${\rm N}_{\rm ts}$ & $10\cdot 2^0$  & $10\cdot 2^1$  & $10\cdot 2^2$  & $10\cdot 2^3$ & $10\cdot 2^4$  & $10\cdot 2^5$ & $10\cdot 2^6$ \\ \hline
		{\rm error} & 3.801$\cdot 10^{-7}$ & 4.893$\cdot 10^{-8}$ & 6.209$\cdot 10^{-9}$ & 7.820$\cdot 10^{-10}$ & 9.820$\cdot 10^{-11}$ & 1.247$\cdot 10^{-11}$ & 1.812$\cdot 10^{-12}$ \\ \hline       {\rm order} & -- & 2.96 & 2.98 & 2.99 & 2.99 & 2.98 & 2.77 \\
		\hline \hline
	\end{tabular}
	\caption{\textit{Convergence rate of the time numerical discretization defined in {Eqs.~\eqref{eq_scheme_3rd}}. The error is calculated in $L^2-$norm, for  different values of  $\varepsilon \in \{ 10^{-2}, 10^{-4}, 10^{-6}\}$. The parameters of the tests are: $t_{\rm fin}=0.1, \Delta t_{\rm ref} = 10^{-5},  N_{\rm ref} = 160, D = 0.02, \delta = 10^{-2}$. The domain is $\Omega = [-1,1]^2\setminus\mathcal{B}$, the initial condition is defined in Eq.~\eqref{eq_expr_IC_tests}  with $x_{m_1} = y_{m_1} = 0$ and $\sigma = 0.1$, and the velocity expression in Eq.~\eqref{eq_expr_velocity}, $A = 1$.}}
	\label{table_time_3rd_order}
\end{table}
\end{landscape}

\section{Conclusions}
\label{sec:conclusions}
In this work, we have presented a recursive technique to construct  numerical schemes that are uniformly accurate in time, of arbitrary order, for highly oscillatory equations. Starting from a numerical scheme of first order, we show how to obtain a time discretization of the desired order of accuracy. This paper is the natural continuation of the work presented in \cite{astuto2023time}, wherein we consider higher order discretization in space and time. Notably, we observe significant improvements in the values of the relative errors, particularly in achieving uniform accuracy, attributable to the adoption of fourth-order spatial discretization. In addition, we intend to implement a novel approach to spatial discretization. Specifically, we aim to adopt a ghost nodal variational formulation, as proposed in \cite{astuto2024nodal}. This formulation will be extended to support high-order accuracy, allowing a comparison with the current spatial discretization approach.

Our primary focus lies in capturing the multiple spatial and temporal scales present in the diffusion and trapping dynamics of a surfactant around a rapidly oscillating trap. This investigation, outlined in \cite{CiCP-31-707, COCO2020109623, ASTUTO2023111880}, aims to investigate the influence of cell membrane oscillations on the capture efficiency of substances diffusing near the cell. In this real-world scenario, there are several orders of magnitude of difference between oscillation amplitudes and frequencies, necessitating the use of high-order numerical schemes to accurately model the phenomenon.

In this study %, we operate under the assumption that the fluid velocity is a known function of space and time, and 
we extend the study to domains of arbitrary shapes, to show the versatility and effectiveness of the proposed numerical schemes. Significant improvements are observed when employing the same numerical scheme for Dirichlet boundary conditions, where we do not need the interpolation of the numerical solution to obtain approximations of its space derivatives at the boundary. In this case, a fourth order accurate numerical scheme in space is obtained.

Expanding on our current research, we aim to explore a two-way coupling mechanism where the motion of the bubble is not explicitly known, but it is influenced by the surrounding fluid dynamics. This coupling introduces a dynamic interplay wherein the bubble behavior is not only influenced by the fluid but also dictates fluid motion. A way to do that, is the coupling of the system with Navier-Stokes equations, as we did in \cite{ASTUTO2023111880}. Furthermore, we want to study the relationship between surface tension and local surfactant concentration, thereby uncovering how variations in surfactant concentration impact bubble dynamics and, consequently, fluid behavior \cite{speirs2018entry}.

\section*{Acknowledgments}
{This work has been supported by the Spoke 10 Future AI Research (FAIR) of the Italian Research Center funded by the Ministry of University and Research as part of the National Recovery and Resilience Plan (PNRR). \\
This work has also been supported also by Italian Ministerial grant PRIN 2022 “Efficient numerical schemes and optimal control methods for time-dependent partial differential equations”, No. 2022N9BM3N - Finanziato dall’Unione europea - Next Generation EU – CUP: E53D23005830006.} \\
{This work has also been supported by the Italian Ministerial grant PRIN 2022 PNRR “FIN4GEO: Forward and Inverse Numerical Modeling of hydrothermal systems in volcanic regions with application to geothermal energy exploitation”, No. P2022BNB97 - Finanziato dall’Unione europea - Next Generation EU – CUP: E53D23017960001. } \\
{The author is a member of the Gruppo Nazionale Calcolo Scientifico-Istituto Nazionale di Alta Matematica (GNCS-INdAM).}

%%%% Bibliography  %%%%%%%%%%

\end{document}